\documentclass{article}
\usepackage[margin=1in]{geometry}

\usepackage[utf8]{inputenc}
\usepackage{amsfonts}
\usepackage{amsmath,amscd,amsthm,amssymb,mathrsfs,setspace}
\usepackage{mathtools}
\usepackage{cite}
\usepackage{algorithm}
\usepackage{algorithmicx}
\usepackage{algpseudocode}
\usepackage{booktabs}
\usepackage{graphicx}
\usepackage{tikz}
\usepackage{subfigure}
\usepackage{url}
\usepackage{fancyhdr}
\usepackage{xcolor}
\usepackage[colorlinks]{hyperref}
\usepackage[nameinlink,capitalize]{cleveref}
\usepackage{siunitx}

\newtheorem{theorem}{Theorem}[section]
\newtheorem{proposition}[theorem]{Proposition}
\newtheorem{lemma}[theorem]{Lemma}

\newtheorem{definition}[theorem]{Definition}
\newtheorem{assumption}[theorem]{Assumption}
\newtheorem{example}[theorem]{Example}

\theoremstyle{remark}
\newtheorem{remark}[theorem]{Remark}

\crefname{theorem}{Theorem}{Theorems}
\Crefname{theorem}{Theorem}{Theorems}
\crefname{assumption}{Assumption}{Assumptions}
\Crefname{assumption}{Assumption}{Assumptions}
\crefname{lemma}{Lemma}{Lemmas}
\Crefname{lemma}{Lemma}{Lemmas}
\crefname{definition}{Definition}{Definitions}
\Crefname{definition}{Definition}{Definitions}
\crefname{proposition}{Proposition}{Propositions}
\Crefname{proposition}{Proposition}{Propositions}
\crefname{algorithm}{Algorithm}{Algorithms}
\Crefname{algorithm}{Algorithm}{Algorithms}
\crefname{section}{Section}{Sections}
\Crefname{section}{Section}{Sections}
\crefname{appendix}{Appendix}{Appendices}
\Crefname{appendix}{Appendix}{Appendices}
\crefname{corollary}{Corollary}{Corollaries}
\Crefname{corollary}{Corollary}{Corollaries}
\crefname{example}{Example}{Examples}
\Crefname{example}{Example}{Examples}

\definecolor{lccx}{HTML}{92268F}

\ifpdf
\hypersetup{
	pdftitle={TVRT2D},
	pdfauthor={A. Schiemann and P. Manns},
	linkcolor=lccx
}
\fi

\DeclareMathOperator*{\argmin}{arg\,min}

\newcommand{\N}{\mathbb{N}}
\newcommand{\R}{\mathbb{R}}
\newcommand{\Z}{\mathbb{Z}}
\newcommand{\Ha}{\mathcal{H}}
\newcommand{\conv}{\operatorname{conv}}

\DeclareMathOperator*{\TV}{TV}
\DeclareMathOperator*{\TVh}{TV^h}

\DeclareMathOperator*{\BV}{BV}
\DeclareMathOperator*{\BVW}{BV_W}

\DeclareMathOperator*{\dvg}{div}

\newcommand*\dd{\mathop{}\!\mathrm{d}}

\newcommand{\weakstarto}{\stackrel{\ast}{\rightharpoonup}}
\newcommand{\eps}{\varepsilon}
\newcommand{\QQ}{\mathcal{Q}}

\newcommand{\EE}{\mathcal{E}}
\newcommand{\Scal}{\mathcal{S}}

\renewcommand{\div}{\dvg}

\title{Discretization of Total Variation in Optimization with Integrality Constraints\thanks{
		The authors acknowledge funding by Deutsche
		Forschungsgemeinschaft (DFG) under grant nos.\ 
		MA 10080/2-1 and ME 3281/10-1.}}

\author{Annika Schiemann\thanks{Faculty of Mathematics,
		TU Dortmund University (\url{annika.schiemann@tu-dortmund.de}, \url{paul.manns@tu-dortmund.de}).}
	\and 
	Paul Manns\footnotemark[2]}

\begin{document}
	\maketitle
	
	\begin{abstract}
		We introduce discretizations of infinite-dimensional optimization problems with total variation regularization and integrality constraints on the optimization variables. We advance the discretization of the dual formulation of the total variation term with Raviart--Thomas functions which is known from literature for certain convex problems. Since we have an integrality constraint, the previous analysis from Caillaud and Chambolle
		\cite{Caillaud2020Error} does not hold anymore. Even weaker $\Gamma$-convergence results do not hold anymore because the recovery sequences generally need to attain non-integer values to recover the total variation of the limit function. We solve this issue by introducing a discretization of the input functions on an embedded, finer mesh.
		A superlinear coupling of
		the mesh sizes implies
		an averaging on the coarser mesh of the Raviart--Thomas ansatz, which enables to recover the total variation of integer-valued limit functions with integer-valued discretized input functions.
		Moreover, we are able to estimate the discretized total variation of the recovery sequence by the total variation of its limit and an error depending on the mesh size ratio. 
		For the discretized optimization problems, we additionally add a constraint that vanishes in the limit and enforces compactness of the sequence of minimizers, which yields their convergence to a minimizer of the original problem. This constraint contains a degree of freedom whose admissible range is determined. Its choice may have a strong impact on the solutions in practice as we demonstrate with an example from imaging.
	\end{abstract}
	
\noindent \textbf{Keywords:}
		total variation; 
		infinite-dimensional
		optimization with integrality
		restrictions;
		coupled discretizations; Raviart--Thomas elements

\noindent \textbf{AMS subject classification:}
		65D18;65K10;49M25
	
	\section{Introduction}
	Let $\Omega \subset \R^d$, $d \in \{1,2,3\}$,
	be a bounded Lipschitz domain, which
	is a finite union of bounded intervals or axis-aligned squares or cubes. We consider discretizations of the optimization problem
	\begin{gather}\label{eq:p}
		\begin{aligned}
			\min_{w \in L^1(\Omega)}\ & F(w) + \alpha \TV(w)\\
			\text{s.t.}\quad & w(x) \in W \subset \Z
			\text{ for almost all (a.a.) } x \in \Omega,
		\end{aligned}\tag{P}
	\end{gather}
	where $W \subset \Z$ is a finite set of integers, $\alpha >0$, $\TV: L^1(\Omega) \to [0,\infty]$ is the total variation, and $F: L^1(\Omega) \to \R$ is continuous and bounded from below.
	For example, $F$ may contain a solution operator to a PDE as in \cite{manns2023integer}, where a trust-region algorithm for integer optimal control problems with total variation regularization was analyzed. Hence, the results of this work can be used to discretize the function space subproblems arising in \cite{manns2023integer}. Integer optimal control problems with total variation regularization have many applications \cite{Gottlich2017Partial,Gottlich2019Partial,Gerdts2005Solving,Hante2017Challenges}.
	Other problems related to \eqref{eq:p} are investigated in the context of image denoising, deblurring \cite{Rudin1992Nonlinear,chambolle1997image,Oliveira2009Adapted,Destuynder2007Dual}, and segmentation \cite{Pock2009convex}. This is because the total variation allows but penalizes jumps of the input function.
	
	One approach, in particular
	for $W = \{0,1\}$, to handle
	the total variation
	in combination with the
	integrality condition
	would be to use a Modica--Mortola energy functional 
	\cite{modica1987gradient}
	and
	relax the
	integrality condition.
	Then, one drives a parameter to zero that
	controls the non-binarity and may recover the total variation in the limit.
	In this case, the difficulty
	of the integrality is replaced by the non-convexity of the 
	Modica--Mortola energy, however.
	
	Therefore, the aim of this paper is to discretize \eqref{eq:p} such that we retain the integrality condition on the input functions while being able to recover the total variation of solutions to \eqref{eq:p}. This is
	challenging since the underlying mesh prescribes the geometry of the discretized input functions.
	Many discretizations of the total variation term can already be found in literature but
	focus rarely on  restrictions to integers, see \cite{chambolle2021approximating} for an overview.
	Approaches that rely on the primal formulation of the total variation are finite-differences discretizations and variants \cite{Wang2011Error,Lai2012Piecewise,Chambolle2011Upwind,Abergel2017Shannon}
	as were introduced for Euler's elastica energy and generalized $\TV$-Stokes models in \cite{Tai2011Fast,Hahn2011Augmented}, $P_1$ discretizations \cite{Bartels2012Total,Bartels2014Discrete}, and non-conforming discretizations with Crouzeix--Raviart functions \cite{Chambolle2020Crouzeix}. A discretization of the dual formulation by discretizing the dual fields with Raviart--Thomas functions is presented in \cite{Caillaud2020Error}. Related approaches
	based on discretized dual fields
	are \cite{Condat2017Discrete,Hintermuller2014Functional,Chambolle2020Learning}. While a $P_1$ or Crouzeix--Raviart input function ansatz
	does not permit the integrality condition, the other approaches allow for piecewise constant ansatz functions. However, if there are $\Gamma$-convergence results for the discretized total variation, the recovery sequences generally need to attain non-integer values to recover the total variation of the limit functions, even if the limit functions themselves are integer-valued, see the proofs of Theorem 4 in \cite{Chambolle2017Accelerated} and Theorem 1.2 in \cite{Caillaud2020Error}. The latter proves an error rate for the discretized optimization problems by means of strong duality of the continuous problems which does
	not hold in our integer setting.
	
	We will solve this issue by introducing two coupled discretizations with a fine mesh for the input functions that is embedded into a coarser mesh for the total variation term. A superlinear coupling of the fine mesh size to the coarser one allows to consider integer-valued and piecewise constant discretized input functions and still be able to recover the total variation of limit functions due to an averaging effect on the coarser mesh. This makes it possible to prove convergence of the discretized total variation in the sense of $\Gamma$-convergence despite the restriction to integrality. Moreover, we are able to prove an error bound for the discretized total variation of the recovery sequence with respect to the total variation of its limit.
	
	The discretization of the considered optimization problems is accompanied with the loss of compactness of the sequence of minimizers, which is due to possible chattering enabled by the enlarged null space of the discretized total variation caused by the coupled meshes. We compensate for this lack by adding a constraint to the discretized problems that enforces compactness on the sequence of minimizers and vanishes in the limit. This constraint contains a degree of freedom whose admissible range we determine. Even though its concrete choice is irrelevant when the mesh sizes are driven to zero, it may impact on the solutions of the discretized problems in practice.
	Together with the approximation properties of the discretized total variation, we prove the convergence of minimizers of the discretized problems to a minimizer of the original problem.
	To solve the discretized problems, we will introduce an outer approximation algorithm, which we apply to an example from imaging.
	
	Structure of the remainder: we start by introducing the discretized total variation and analyzing its approximation properties in \cref{sec:Preliminary}. \Cref{sec:Gamma-convergence} is dedicated to the discretization of the considered optimization problems and the convergence of their minimizers to a minimizer of the original problem. In \cref{sec:OA}, we state and analyze an outer-approximation algorithm to solve the discretized problems. We present our numerical experiments and results in \cref{sec:Numeric} and conclude in \cref{sec:Outlook}.
	\section{Discretized total variation} \label{sec:Preliminary}
	The space
	of functions of bounded variation $\BV(\Omega)$ is the space of 
	functions $w \in L^1(\Omega)$ whose total variation is finite, that is,
	\[ \TV(w) \coloneqq \sup
	\left\{ \int_\Omega w(x) \dvg \phi(x)\dd x\,\middle|\,
	\phi \in C^1_c(\Omega;\R^d)\ \text{ and }\ \sup_{\mathclap{x \in \Omega}} \|\phi(x)\| \le 1
	\right\} < \infty.
	\]
	We denote the feasible set of \eqref{eq:p} by $L^1_W(\Omega) \coloneqq \left\{ w \in L^1(\Omega) \middle| \, w(x) \in W \text{ for a.a.\ } x \in \Omega \right\}$
	and its subset of solutions with finite objective value by $\BVW(\Omega) \coloneqq L^1_W(\Omega) \cap \BV(\Omega)$.
	\begin{lemma}\label{lem:convergence_in_L^p}
		Let $\{w^h\}_{h>0} \subset L^1_W(\Omega)$ be given such that $w^h \to w$ in $L^1(\Omega)$ as $h \searrow 0$. Then we have that $w \in L^1_W(\Omega)$ and $w^h \to w$ in $L^p(\Omega)$ for all $1 \leq p < \infty$.
	\end{lemma}
	\begin{proof}
		The first claim follows because $w^h \to w$  pointwise a.e.\ for a subsequence, cf. Lemma 2.2 in \cite{leyffer2021sequential}. The second claim follows because
		\begin{align*}
			\| w- w^h \|_{L^p(\Omega)}^p \leq \|(w-w^h)^{p-1} \|_{L^\infty(\Omega)} \| w - w^h \|_{L^1(\Omega)} \to 0 \text{ as } h \searrow 0,
		\end{align*}
		since $\| (w - w^h)^{p-1} \|_{L^\infty(\Omega)} \leq (\max_{w_1,w_2 \in W} |w_1-w_2|)^{p-1}$ for all $h >0$.
	\end{proof}
	By the following lemma, which was proven in Theorem 1 in \cite{hintermuller2015density}, we can replace the test function space $C_c^1(\Omega;\R^d)$ for the computation of the total variation term by $H_0(\div;\Omega)$. Since $\Omega$ is a Lipschitz domain, we may define $H_0(\div;\Omega)$ by
	\begin{align*}
		H_0(\div;\Omega) \coloneqq \left\{ v \in H(\div;\Omega) \, \middle| \, v \cdot n |_{\partial \Omega} \equiv 0 \right\},
	\end{align*}
	where $n$ denotes the outer unit normal of $\Omega$ and $H_0(\div;\Omega) = H^1_0(\Omega)$ for $d=1$.
	\begin{lemma}\label{lem:TV_div}
		For all $w \in L^2(\Omega)$, it holds that
		\[
		\TV(w) = \sup \left\{ \int_\Omega w(x) \dvg \phi(x)\dd x\,\middle|\,
		\phi \in H_0(\dvg ; \Omega) \ \text{ and }\ \|\phi\|_{L^\infty(\Omega;\R^d)} \le 1 \right\}.
		\]
	\end{lemma}
	In order to discretize the total variation and later problem \eqref{eq:p}, we consider finite element meshes on $\Omega$ which fulfill the following assumptions.
	\begin{assumption}\label{assu:grid}\mbox{~}
		\begin{enumerate}
			\item $\TV$ meshes: For $h>0$, we consider a partition $\QQ_h$ of $\Omega$ into intervals or axis-aligned squares or cubes $Q \in \QQ_h$ of height $h>0$, that is, $\Omega = \bigcup_{ Q \in \QQ_h} Q$.\label{assu:TV_mesh}
			\item Input meshes: For each $h > 0$, we consider a mesh $\QQ_{\tau_h}$ of $\Omega$ of intervals or axis-aligned squares or cubes $Q \in \QQ_{\tau_h}$ of height $\tau_h \in (0,h]$ that is embedded into the $\TV$ mesh. Specifically, for each $\tilde{Q} \in \QQ_h$ there exists $\QQ_{\tilde{Q}} \subset \QQ_{\tau_h}$ such that $\bigcup_{Q \in \QQ_Q} Q = \tilde{Q}$. \label{assu:control_mesh}
		\end{enumerate}
	\end{assumption}
	We introduce the discretized total variation in \cref{subsec:Discretization_TV}. We prove its approximation properties in a $\Gamma$-convergence sense below, where \cref{subsec:Liminf_TV} is dedicated to the $\liminf$ inequality and \cref{subsec:Limsup_TV} provides the $\limsup$ inequality and an error estimate for the discretized total variation of the recovery sequence.
	
	\subsection{Discretization of the total variation}\label{subsec:Discretization_TV}
	
	For a mesh $\QQ_h$ with $h >0$ fulfilling \cref{assu:grid}.\ref{assu:TV_mesh} and $w \in \BV(\Omega)$, we define the following discretization of the total variation term $\TV$ as in \cite{Caillaud2020Error} by replacing the test functions for the computation of $\TV$ by Raviart--Thomas functions, that is,
	\[ {\TV}^{h}(w) \coloneqq \sup \left\{ \int_\Omega w(x)\dvg \phi(x)\dd x
	\,\middle|\,
	\phi \in RT0_0^h \ \text{ and }\ \|\phi\|_{L^\infty(\Omega,\R^d)} \le 1
	\right\},
	\]
	where $RT0^h_0 \coloneqq \left\{ \phi \in RT0^h \, \middle| \, (\phi \cdot n) |_{\partial \Omega} \equiv 0 \right\}$ and 
	$n$ again denotes the outer normal of $\Omega$.
	The space $RT0^h \subset H(\div;\Omega)$ is the lowest-order Raviart--Thomas space defined on the mesh $\QQ_h$ which is, according to \cite{brezzi1991mixedfe}, defined as follows.
		For each element $Q \in \QQ_h$, we define the space of polynomials of degree $k_i$ in $x_i$ on $Q$ by
		\begin{align*}
			P^h_{k_1}(Q) &\coloneqq \left\{ p : Q \to \R \, : \, p(x) = \sum_{i \leq k} a_i x^i \right\}, \\
			P^h_{k_1,k_2}(Q) &\coloneqq \left\{ p : Q \to \R \, : \, p(x_1,x_2) = \sum_{i \leq k_1, j \leq k_2} a_{ij} x_1^i x_2^j \right\}, \\
			P^h_{k_1,k_2,k_3}(Q) & \coloneqq \left\{ p : Q \to \R \, : \, p(x_1,x_2,x_3) = \sum_{i \leq k_1,j \leq k_2, \ell \leq k_3} a_{ij \ell} x_1^i x_2^j x_3^\ell \right\}.
		\end{align*}
		We highlight that $P0^h(Q) = P^h_0(Q)$ if $d=1$, $P0^h(Q) = P^h_{0,0}(Q)$ if $d=2$, and $P0^h(Q) = P^h_{0,0,0}(Q)$ if $d=3$.
		We define for each element $Q \in \QQ_h$
		\begin{align*}
			RT0^h(Q) \coloneqq \begin{cases}
				P^h_1(Q) & \text{ if } d = 1,\\
				P^h_{1,0}(Q) \times P^h_{0,1}(Q) & \text{ if } d=2, \\
				P^h_{1,0,0}(Q) \times P^h_{0,1,0}(Q) \times P^h_{0,0,1}(Q) & \text{ if } d=3.
			\end{cases}
	\end{align*}
	In particular, this implies that the elements in $RT0^h(Q)$ are linear in $Q$. The space $RT0^h$ of lowest-order Raviart--Thomas functions is then defined by
		\begin{align*}
			RT0^h \coloneqq \left\{ \phi \in H(\div;\Omega) \, : \, \phi |_Q \in RT0^h(Q) \, \text{
				for all } \, Q \in \QQ_h \right\}.
		\end{align*}
		Note that the restriction to axis-aligned squares and cubes $Q$ is mandatory for the above definition of Raviart--Thomas functions.
		
		By § III.3.3 in \cite{brezzi1991mixedfe}, there exists an interpolation operator $I_{RT0^h} : W^{1,\infty}(\Omega; \R^d) \to RT0^h$. By Proposition 3.7 in \cite{brezzi1991mixedfe} there holds for all $\phi \in W^{1,\infty}(\Omega;\R^d)$ that
		\begin{align}
			\div(I_{RT0^h} \phi) = \Pi_{P0^h} \div \phi. \label{eq:interpolation_div}
		\end{align}
		By Proposition 3.8 in \cite{brezzi1991mixedfe}, there exists a constant $b>0$ such that
		\begin{align}
			\| \div(\phi - I_{RT0^h} \phi) \|_{L^2(\Omega)} \leq b h \| \nabla \div \phi \|_{L^2(\Omega;\R^d)} \label{eq:est_norm_interpolation}
		\end{align}
		for all $\phi \in W^{1,\infty}(\Omega;\R^d)$ with $\dvg \phi \in H^1(\Omega)$.
		Let $\phi \in W^{1, \infty}(\Omega;\R^d)$. There exists a constant $C(\phi) >0$ such that
		\begin{align}
			\| I_{RT0^h}\phi\|_{L^\infty(\Omega;\R^d)} \leq \| \phi \|_{L^\infty(\Omega;\R^d)} + C(\phi)h. \label{eq:RT_linf_error_estimate}
		\end{align}
		The above statement can be found in Lemma 3.6 in \cite{Caillaud2020Error} for the case $d=2$. The proof can be adapted straightforwardly to the cases $d=1$ and $d=3$.
	\begin{lemma}\label{lem:div_bound}
		Let $\phi \in RT0^h$ for some $h >0$ satisfy $\| \phi \|_{L^\infty(\Omega;\R^d)} \leq 1$. Then $\| \dvg \phi \|_{L^\infty(\Omega)} \leq \frac{2d}{h}$.
	\end{lemma}
	\begin{proof}
		Consider $\phi |_{\bar{Q}}$ on the element $\bar{Q} = l+ [0,h]^d$ with $l \in \R^d$. Then $\phi |_{\bar{Q}} (x) = a + \sum_{i=1}^{d} c_i e^i x$ for $x \in \bar{Q}$ with $a,c \in \R^d$ and $e^i \in \R^d$, $i \in \{1,\dots,d\}$, the canonical unit vector basis. Since $\| \phi \|_{L^\infty(\Omega;\R^d)} \leq 1$, we have for each $i \in \{1,...,d\}$ that $|a_i + c_i x_i | \leq 1$ for all $x_i \in l_i + [0,h]$, in particular $|a_i + c_i l_i + c_ih | \leq 1 \text{ and } |a_i + c_i l_i | \leq 1.$
		This yields
		\begin{align*}
			| \div \phi |_{\bar{Q}} | \leq \sum_{i=1}^d | c_i | \leq \sum_{i=1}^d \frac{1}{h} ( |a_i + c_i l_i + c_i h | + |a_i +c_i l_i|) \leq \sum_{i=1}^d \frac{2}{h} = \frac{2d}{h}.
		\end{align*}
	\end{proof}
	In contrast to $\TV$, the discretized total variation $\TVh$ always admits a maximizer and is thus always finite. Moreover, $\TV$ is always an upper bound for $\TVh$.
	\begin{lemma}\label{lem:TVh}
		Let $w \in L^1_W(\Omega)$. Then there is some $\phi \in RT0^h_0$ with $\| \phi \|_{L^\infty(\Omega;\R^d)} \leq 1$ such that $\TVh(w) = \int_{\Omega} w(x) \div \phi(x ) \dd x < \infty$ and $\TVh(w) \le \TV(w)$.
	\end{lemma}
	\begin{proof}
		$\TVh(w) \le \TV(w)$ follows from \cref{lem:TV_div} since $RT0_0^h \subset H_0(\div;\Omega)$. The existence of a maximizer of $\TVh$ follows from the fact that the corresponding maximization problem can be rewritten as a finite-dimensional optimization problem with linear objective function and compact feasible set. This is because $RT0_0^h$ is finite-dimensional,
		$\div \phi$ for $\phi \in RT0_0^h$ is constant on the grid cells $Q \in \QQ_h$, and the feasible set of $\TVh(w)$ can be described by finitely many convex inequalities by bounding the Euclidean norm of the point evaluations of $\phi$ in each node of the mesh $\QQ_h$ by $1$.
	\end{proof}
	In order to discretize \eqref{eq:p}, we replace the total variation term $\TV(w)$ by the discretized total variation term $\TVh(w)$ in the objective. As a result, we can no longer guarantee the existence of minimizers, as stated in the following remark, because $\TVh$ has a greater null space than $\TV$. We will fix this issue in \cref{sec:Gamma-convergence}.
	\begin{remark}\label{expl:existence_of_minimizers}
		If we replace $\TV$ by $\TVh$ in problem \eqref{eq:p}, we can not guarantee the existence of minimizers anymore because $\TVh$ has a greater null space than $\TV$. This is due to the fact that for each $w \in L^1(\Omega)$, we have that
		\begin{align}\label{eq:TVh_P0h}
			\int_\Omega w(x) \div \phi (x) \dd x = \sum_{Q \in \QQ_h} d_Q \int_{Q} w(x) \dd x = \int_\Omega \tilde{w}(x) \div \phi (x) \dd x
		\end{align}
		for all $\phi \in RT0^h$ with $\div \phi |_{Q} \eqqcolon d_Q \in \R$ for all $Q \in \QQ_h$ and all $\tilde{w} \in L^1(\Omega)$ with $\int_Q \tilde{w}(x) \dd x = \int_{Q} w(x) \dd x$ for all $Q \in \QQ_h$. This in particular yields $\TVh(w) = \TVh(\tilde{w})$. Hence, there may exist a minimizing sequence $\{w^n\}_{n \in \N} \subset \BVW(\Omega)$ consisting of chattering functions (for example with checkerboard structure) with $\TVh(w^n) =0$ for all $n \in \N$ that converges weakly to some limit function that is not in $\BVW(\Omega)$.
	\end{remark}
	\subsection{Lim inf inequality for $\TVh$} \label{subsec:Liminf_TV}
	Let a mesh $\QQ_\tau$ with $\tau>0$ fulfill \cref{assu:grid}. We define the space of functions that are piecewise constant on the mesh cells $Q \in \QQ_{\tau}$ by $P0^\tau$.
	Let $\Pi_{P0^\tau}: L^1(\Omega) \to P0^\tau$ denote the projection onto $P0^\tau$.
	\begin{lemma}\label{lem:estimate_ProjP0h}
		Let $w \in L^1_W(\Omega)$. Then $\| w - \Pi_{P0^\tau}w\|_{L^1(\Omega)} \to 0 \text{ as } \tau \searrow 0$.
		Moreover, for all $w \in \BVW(\Omega)$ there holds $\|w - \Pi_{P0^\tau}w\|_{L^1(\Omega)} \le \sqrt{d} \tau \TV(w)$.
	\end{lemma}
	\begin{proof}
		Let $w \in L^1_W(\Omega)$. Lebesgue's differentiation theorem \cite[Chap. 3, Cor. 1.6 and 1.7]{stein2005real} gives $\Pi_{P0^\tau}w \to w$ pointwise a.e.\ in $\Omega$ as $\tau \searrow 0$. Moreover, $\| \Pi_{P0^\tau}w\|_{L^\infty(\Omega)} \leq \max_{w \in W} |w| \in \R$ such that Lebesgue's dominated convergence theorem yields $\| w - \Pi_{P0^\tau}w \|_{L^1(\Omega)} \to 0$ as $\tau \searrow 0$.
		Since $\Omega$ is in particular of finite perimeter, the proof of the second claim runs along the lines of Theorem 12.26, particularly (12.24),
		in \cite{maggi2012sets}. A similar proof for $\Omega = (0,1)^2$ is given in Lemma 3.2 in \cite{Caillaud2020Error}.
	\end{proof}
	\begin{lemma}
		Let $\{w^h\}_{h>0} \subset L^1(\Omega)$ with $w^h \to w$ in $L^1(\Omega)$ and $w \in L^2(\Omega)$.
		Let $\phi \in C^1_c(\Omega;\R^d)$ satisfy $\|\phi\|_{L^\infty(\Omega;\R^d)} \le 1$. Then
		\begin{align} \label{eq:div_liminf_TVh}
			\int_{\Omega} \dvg \phi(x) w(x)\dd x 
			\le \liminf_{h \searrow 0} \TVh(w^h).
		\end{align}
	\end{lemma}
	\begin{proof}
		Let $\varepsilon \in (0,1)$. Then we approximate $\phi$ with $\hat{\phi} \coloneqq (1-\eps) \phi \in C_c^1(\Omega;\R^d)$
		such that $\|\hat{\phi}\|_{L^\infty(\Omega;\R^d)} = (1-\eps) \| \phi \|_{L^\infty(\Omega;\R^d)} \le 1 - \varepsilon$ and
		\begin{align*}
			\int_{\Omega} \dvg \phi(x) w(x)\dd x & = \int_\Omega \div \hat{\phi}(x) w(x) \dd x + \eps \int_\Omega \div \phi (x) w(x) \dd x \\
			& \le \int_{\Omega} \dvg \hat{\phi}(x) w(x)\dd x + \underbrace{\| \div \phi \|_{L^\infty(\Omega)} \|w\|_{L^1(\Omega)}}_{c_1\coloneqq } \varepsilon.
		\end{align*}
		By virtue of the error estimate \eqref{eq:RT_linf_error_estimate}, we obtain
		that $\| I_{RT0^h} \hat{\phi}\|_{L^\infty(\Omega;\R^d)} \le 1$ holds for all $h \leq \frac{\eps}{C(\hat{\phi})}$. Moreover, since $\hat{\phi}$ has compact support in $\Omega$, there holds $\hat{\phi} \cdot n \equiv 0$ on $\partial \Omega$, where $n$ denotes the unit outer normal of $\Omega$, so that $I_{RT0^h}\hat{\phi} \in RT0^h_0$.
		Since $\hat{\phi} \in C^1_c(\Omega;\R^d)$, there exists a constant $M(\hat{\phi})>0$ only depending on $\hat{\phi}$ such that $\| \div \hat{\phi} \|_{L^\infty(\Omega)} \leq M(\hat{\phi})$.
		This implies that $\| \Pi_{P0^h} \dvg\hat{\phi}\|_{L^\infty(\Omega)} \leq M(\hat{\phi})$ for all $h >0$ such that
		\begin{align*}
			\hspace{2em}&\hspace{-2em} \int_{\Omega} \dvg I_{RT0^h} \hat{\phi}(x) w(x)\dd x \\
			&\le \int_{\Omega} \dvg I_{RT0^h} \hat{\phi}(x) w^h(x)\dd x
			+ \|\dvg I_{RT0^h}\hat{\phi}\|_{L^\infty(\Omega)}\|w^h - w\|_{L^1(\Omega)}  \\
			&\le \TVh(w^h) + \|\dvg I_{RT0^h}\hat{\phi}\|_{L^\infty(\Omega)}\|w^h - w\|_{L^1(\Omega)} \\
			&= \TVh(w^h) + \| \Pi_{P0^h} \dvg\hat{\phi}\|_{L^\infty(\Omega)}\|w^h - w\|_{L^1(\Omega)} \\
			&\leq \TVh(w^h) + M(\hat{\phi})\|w^h - w\|_{L^1(\Omega)},
		\end{align*}
		where we used $\div I_{RT0^h} \hat{\phi} = \Pi_{P0^h} \div \hat{\phi}$ by \eqref{eq:interpolation_div}, yielding
		\begin{align*}
			\hspace{2em}&\hspace{-2em} \int_{\Omega} \dvg \phi(x) w(x)\dd x \leq \int_\Omega \div \hat{\phi}(x) w(x) \dd x + c_1 \eps \\
			& =  \int_\Omega \div I_{RT0^h}\hat{\phi}(x) w(x) \dd x + \int_\Omega \div(\hat{\phi} - I_{RT0^h } \hat{\phi})(x) w(x) \dd x + c_1 \eps \\
			& \leq \TVh(w^h) + M(\hat{\phi}) \| w^h - w \|_{L^1 (\Omega)} + \| \div(\hat{\phi}- I_{RT0^h} \hat{\phi}) \|_{L^2(\Omega)} \|w \|_{L^2(\Omega)} + c_1 \eps.
		\end{align*}
		Again using \eqref{eq:interpolation_div}, it follows that
			\begin{align*}
				\| \div(\hat{\phi}- I_{RT0^h} \hat{\phi}) \|_{L^2(\Omega)} = \| \div \hat{\phi}- \Pi_{P0^h} \div \hat{\phi} \|_{L^2(\Omega)} \to 0
			\end{align*}
		as $h \searrow 0$.
		Since $w \in L^2(\Omega)$, driving $h \searrow 0$ implies that
		\begin{align*}
			\int_{\Omega} \dvg \phi(x) w(x)\dd x \leq c_1 \varepsilon + \liminf_{h \searrow 0} \TVh(w^h).
		\end{align*}
		The claim follows because $\varepsilon > 0$ was chosen arbitrarily.
	\end{proof}
	\begin{theorem} \label{thm:TV_liminf}
		Let $\{w^h\}_{h >0} \subset L^1(\Omega)$ with $w^{h} \to w$ in $L^1(\Omega)$ as $h \searrow 0$ and $w \in L^2(\Omega)$. Then
		\begin{align*}
			\TV(w) \leq \liminf_{h \searrow 0} \TVh(w^{h}).
		\end{align*}
	\end{theorem}
	\begin{proof}
		We supremize over all $\phi \in C_c^1(\Omega;\R^d)$ with $\| \phi \|_{L^\infty(\Omega;\R^d)} \leq 1$ in \eqref{eq:div_liminf_TVh}.
	\end{proof}
	Up to this point the discretization of the input function $w$ has not been relevant. Indeed, the previous statements hold for arbitrary sequences with $w^h \to w$ in $L^1(\Omega)$ with $w \in L^2(\Omega)$.
	Instead, the discretized total variation $\TVh$ discretizes its inputs $w$ implicitly in the sense that $\TVh(w) = \TVh(\Pi_{P0^h}w)$, see \eqref{eq:TVh_P0h}.
	That means if we would waive the integrality constraint for the discretized problems, we could approximate the total variation of a given function $w \in \BVW(\Omega)$ with the discretized total variation of the projections $\Pi_{P0^h} w$. In particular, we refer to the results in \cite{Caillaud2020Error}.
	Since we additionally---and in particular differing from \cite{Caillaud2020Error}---require that the discretized input functions only attain values in the discrete set $W$, we can not work with the projections $\Pi_{P0^h} w$ as in \cite{Caillaud2020Error}.
	Due to the given geometry of the mesh, it is generally not possible to approximate the term $\TV(w)$ for $w \in \BVW(\Omega)$ with $\TVh$ of functions in $P0^h \cap \BVW(\Omega)$, as the following example shows.
	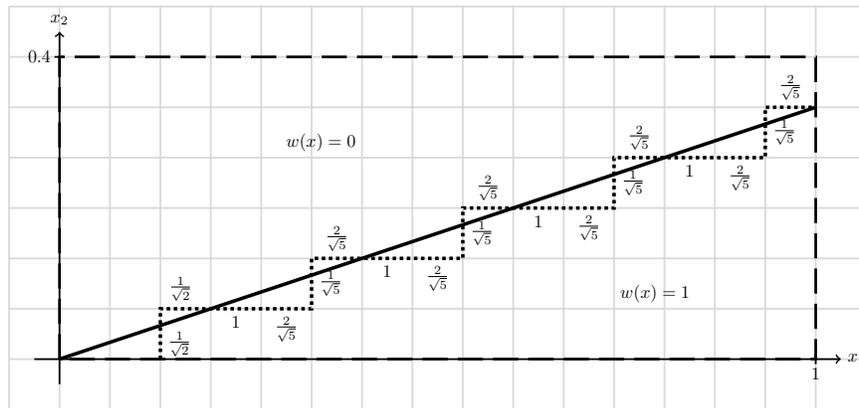
\begin{figure}[h]
		\centering
		\begin{minipage}{\textwidth}
			\centering
			\resizebox{0.7\textwidth}{!}{
				\begin{tikzpicture}
					\def\n{17} 
					\def\m{8} 
					\draw[step=1,gray!30,line width=1pt] (0,0) grid (\n,\m);
					\draw[->, line width=1pt] (\n/2 - 8, \n/2 - 7.5) -- (\n-0.5,\n/2 - 7.5) node[right] {$x_1$};
					\draw[->, line width=1pt] (\n/2 - 7.5, \n/2 - 8) -- (\n/2 - 7.5,\m-0.5) node[above] {$x_2$};
					\draw[line width=1pt] (\n/2 + 7.5,\n/2 - 7.5-.1) -- (\n/2 + 7.5, \n/2 - 7.5+ .1) node[below=4pt] {$1$};
					\draw[line width=1pt] (\n/2 - 7.5-.1, \m-1) -- (\n/2 - 7.5+ .1, \m-1) node[left=4pt] {$0.4$};
					%
					\draw[line width=1.5pt,dash pattern=on 5mm off 2mm] (\n/2 - 7.5, \n/2 - 7.5) rectangle (\n/2 + 7.5, \m/2 + 3);
					%
					\draw[black, line width=2pt] (\n/2 - 7.5, \n/2 - 7.5) -- (\n/2 + 7.5, \n/2 - 7.5/3);
					\draw[line width=2pt, dotted] (\n/2 - 5.5, \n/2 - 7.5) node[right,yshift=0.3cm]{$\frac{1}{\sqrt{2}}$}
					-- (\n/2 - 5.5, \n/2 - 6.5) node[above,xshift=0.4cm]{$\frac{1}{\sqrt{2}}$}
					-- (\n/2 - 4.5, \n/2 - 6.5) node[below,xshift=0.5cm]{$1$}
					-- (\n/2 - 3.5, \n/2 - 6.5) node[below,xshift=0.5cm]{$\frac{2}{\sqrt{5}}$}
					-- (\n/2 - 2.5, \n/2 - 6.5) node[right,yshift=0.5cm]{$\frac{1}{\sqrt{5}}$}
					-- (\n/2 - 2.5, \n/2 - 5.5) node[above,xshift=0.5cm]{$\frac{2}{\sqrt{5}}$}
					-- (\n/2 - 1.5, \n/2 - 5.5) node[below,xshift=0.5cm]{$1$}
					-- (\n/2 - 0.5, \n/2 - 5.5) node[below,xshift=0.5cm]{$\frac{2}{\sqrt{5}}$}
					-- (\n/2 + 0.5, \n/2 - 5.5) node[right,yshift=0.5cm]{$\frac{1}{\sqrt{5}}$}
					-- (\n/2 + 0.5, \n/2 - 4.5) node[above,xshift=0.5cm]{$\frac{2}{\sqrt{5}}$}
					-- (\n/2 + 1.5, \n/2 - 4.5) node[below,xshift=0.5cm]{$1$}
					-- (\n/2 + 2.5, \n/2 - 4.5) node[below,xshift=0.5cm]{$\frac{2}{\sqrt{5}}$}
					-- (\n/2 + 3.5, \n/2 - 4.5) node[right,yshift=0.5cm]{$\frac{1}{\sqrt{5}}$}
					-- (\n/2 + 3.5, \n/2 - 3.5) node[above,xshift=0.5cm]{$\frac{2}{\sqrt{5}}$}
					-- (\n/2 + 4.5, \n/2 - 3.5) node[below,xshift=0.5cm]{$1$}
					-- (\n/2 + 5.5, \n/2 - 3.5) node[below,xshift=0.5cm]{$\frac{2}{\sqrt{5}}$}
					-- (\n/2 + 6.5, \n/2 - 3.5) node[right,yshift=0.5cm]{$\frac{1}{\sqrt{5}}$}
					-- (\n/2 + 6.5, \n/2 - 2.5) node[above,xshift=0.5cm]{$\frac{2}{\sqrt{5}}$}
					-- (\n/2 + 7.5, \n/2 - 2.5);
					\node[above left] at (7,5) {$w(x) = 0$};
					\node[above right] at (12,2) {$w(x) =1$};
			\end{tikzpicture}}
			\caption{Example for the construction in \cref{expl:TVh} with $k = 1$. The level sets of the limit function $w$ are separated by the solid line and the level sets of the rounding of $\Pi_{P0^h} w$ are separated by the dotted line. Nonzero fluxes of $\phi$ are indicated next to the corresponding edges of the grid cells.}
			\label{fig:TVh}
		\end{minipage}
	\end{figure}
	\begin{example}\label{expl:TVh}
		Consider $\Omega = (0,1) \times (0,0.4)$ and the mesh $\QQ_h$ consisting of squares of size $h = \frac{1}{15k}$ for $k \in \N$. Define $w = \chi_{\{(\frac{1}{3},-1)^Tx \geq 0\}}$ and consider $w^h \in P0^h \cap \BVW(\Omega)$ obtained by rounding the projection $\Pi_{P0^h}w$ on each square to the nearest value in $W$, see \cref{fig:TVh} for an example for $k = 1$. Let $\phi \in RT0^h$ be the Raviart--Thomas function that is defined by the fluxes over the edges of the grid cells as exemplarily illustrated in \cref{fig:TVh}. It is immediate that $\| \phi \|_{L^\infty(\Omega;\R^d)} \leq 1$.
		Denote by $\EE_h$ the interior edges of the mesh $\QQ_h$ and let $n_E$ be a unit normal vector to an edge $E \in \EE_h$. Let $w^h_{E^+}$ and $w^h_{E^-}$ denote the values of $w^h$ on the two cubes adjoining $E$. Then,
		\begin{align*}
			\TVh (w^h) & \geq \int_{\Omega} \div \phi \, w^h \dd x = \sum_{E \in \EE_h} | w^h_{E^+} - w^h_{E^-} | \, \int_E \phi \cdot n_E \dd \Ha^{1}(x) \\
			& = \frac{1}{15k} \left( \left( 5k-1 \right) \left( 1 + 2 \cdot \frac{2}{\sqrt{5}} + \frac{1}{\sqrt{5}} \right) + 2 \cdot \frac{1}{\sqrt{2}} \right) \\
			& = \frac{1 + \sqrt{5}}{3} + \frac{\sqrt{2} -1 - \sqrt{5}}{15k} \to \frac{1 + \sqrt{5}}{3} \approx 1.078 \quad \text{as } k \to \infty.
		\end{align*}
		On the other hand, $\TV(w) = \frac{\sqrt{10}}{3} \approx 1.054$ such that $\TVh (w^h) \not \to \TV(w)$ as $h \searrow 0$. Instead, there holds $\TVh(w^h) > \TV(w)$ for $k \geq 5$.
	\end{example}
	\subsection{Lim sup inequality for $\TVh$} \label{subsec:Limsup_TV}
	To achieve an averaging effect by other means than the projection $\Pi_{P0^h}$,
	we discretize the functions $w \in \BVW(\Omega)$ on a finer mesh $\QQ_{\tau_h}$ embedded into the mesh $\QQ_h$ for $\TVh$ such that \cref{assu:grid} is fulfilled. This allows to recover $\TV(w)$ with $\TVh(w^{\tau_h})$ with functions $w^{\tau_h} \in P0^{\tau_h} \cap \BVW(\Omega)$ when the mesh sizes are superlinearly coupled, that is, $\frac{\tau_h}{h} \searrow 0$ as $h \searrow 0$.
	\begin{definition}\label{dfn:RWP0tau}
		We define the operator $R_{P0^\tau}^W : \BVW(\Omega) \to \BVW(\Omega) \cap P0^\tau$ for $\tau >0$ and a corresponding mesh $\QQ_\tau$ as follows.
		For $w \in \BVW(\Omega)$ and $Q \in \QQ_\tau$, let
		\begin{align*}
			R_{P0^\tau}^W(v)|_Q \in \argmin \left\{ \left| \frac{1}{|Q|} \int_Q w(x) \dd x - \omega \right| \, \middle| \, \omega \in W \text{ and } \left|w|_Q^{-1}(\omega) \right| >0\right\}.
		\end{align*}
		If the minimizer is not unique, we choose the smallest one.
	\end{definition}
	Similar to \cref{lem:estimate_ProjP0h}, we obtain convergence results for the sequence $\{R_{P0^\tau}^W(w)\}_{\tau>0}$ for $\tau \searrow 0$ with a convergence rate depending on $\TV(w)$ if $w \in \BVW(\Omega)$.
	\begin{lemma}\label{lem:estimate_RWP0h}
		Let $w \in L^1_W(\Omega)$. Then
		\begin{align*}
			\|w - R^W_{P0^\tau}(w)\|_{L^1(\Omega)} \leq 2 \| w - \Pi_{P0^\tau}w \|_{L^1(\Omega)} \to 0 \text{ as } \tau \searrow 0.
		\end{align*}
		Moreover, for all $w \in \BVW(\Omega)$ the following estimates hold:
		\begin{align*}
			\|w - R^W_{P0^\tau}(w)\|_{L^1(\Omega)} \leq 2 \sqrt{d} \tau \TV(w)
			\:
			\text{ and }
			\:
			\|R^W_{P0^\tau}(w) - \Pi_{P0^\tau}w\|_{L^1(\Omega)} \le \sqrt{d} \tau \TV(w).
		\end{align*}
	\end{lemma}
	\begin{proof}
		For the first claim, let $w \in L^1_W(\Omega)$. By \cref{dfn:RWP0tau}, there holds $\| \Pi_{P0^\tau}w - R^W_{P0^\tau}(w) \|_{L^1(\Omega)} \leq \| w - \Pi_{P0^\tau}w \|_{L^1(\Omega)}$
		due to $| (\Pi_{P0^\tau}w)(x) - (R^W_{P0^\tau}(w))(x) |$ $\leq | (\Pi_{P0^\tau}w)(x) - w(x)|$ and $w(x) \in W$ for a.a.\ $x \in \Omega$ because the rounding operator rounds the value of $\Pi_{P0^\tau}w$ on each cell to the nearest value in $W$ that is also attained by $w$ on the corresponding cell.
		Hence,
		\begin{align*}
			\| w - R^W_{P0^\tau}(w) \|_{L^1(\Omega)} & \leq \| w - \Pi_{P0^\tau}w \|_{L^1(\Omega)} + \| \Pi_{P0^\tau}w - R^W_{P0^\tau}(w) \|_{L^1(\Omega)} \\
			& \leq 2 \| w - \Pi_{P0^\tau}w \|_{L^1(\Omega)} \to 0 \text{ as } \tau \searrow 0.
		\end{align*}
		The second claim follows then by \cref{lem:estimate_ProjP0h}.
	\end{proof}
	With these preparations, we are now able to prove a result corresponding to Lemma 3.1 in \cite{Caillaud2020Error}, where it is stated that $\TVh(\Pi_{P0^h} w) \leq \TV(w)$ for $w \in L^2(\Omega)$. Note that this proof hinges on the fact that $\dvg \phi$ is constant, which is due to our specific choice of lowest-order Raviart--Thomas functions.
	\begin{proposition}\label{lem:estimate_TVhRP0hTV}
		Let $w \in \BVW(\Omega)$ and a tuple $(h,\tau_h)$ with $h >0$ be given such that \cref{assu:grid} is fulfilled. Then
		\begin{align}
			\TVh(R^W_{P0^{\tau_h}}(w)) \le \TVh(w) + \varepsilon(\tau_h,h,w) 
			\le \TV(w) + \varepsilon(\tau_h,h,w) \label{eq:TVh_R_ineq}
		\end{align}
		with $\eps(\tau_h,h,w) \coloneqq \sqrt{d} \TV(w) \frac{2 d \tau_h}{h}$.
	\end{proposition}
	\begin{proof}
		Let $w \in \BVW(\Omega)$. We define $w^{\tau_h} \coloneqq R^W_{P0^{\tau_h}}(w) \in \BVW(\Omega)$ and $\bar{w}^{\tau_h} \coloneqq \Pi_{P0^{\tau_h}}w \in \BV(\Omega)$. 
		Let $\phi \in RT0_0^h$ with $\|\phi\|_{L^\infty(\Omega;\R^d)} \le 1$. Then
		\begin{align*}
			\int_{\Omega} w^{\tau_h}(x) \dvg \phi(x)\dd x
			& = \int_{\Omega}\bar{w}^{\tau_h}(x) \dvg \phi(x)\dd x + \int_{\Omega} (w^{\tau_h}(x) - \bar{w}^{\tau_h}(x)) \dvg \phi(x) \dd x \\
			& \leq \TVh(w) + \|w^{\tau_h} - \bar{w}^{\tau_h} \|_{L^1(\Omega)} \| \div \phi \|_{L^\infty(\Omega)}.
		\end{align*}
		Here, we have used that the meshes $\QQ_h$ and $\QQ_{\tau_h}$ fulfill \cref{assu:grid}, that $\div \phi$ is
		constant on each $Q \in \QQ_h$, and that $\int_Q w(x) \dd x = \int_Q \bar{w}^{\tau_h}(x) \dd x$ for each $Q \in \QQ_h$
		to deduce $\TVh(w) = \TVh(\bar{w}^{\tau_h})$ in the
			derivation of the inequality.
		By \cref{lem:div_bound}, it follows that $\|\dvg \phi\|_{L^\infty(\Omega)} \le \frac{2d}{h}$, and \cref{lem:estimate_RWP0h} yields $\|w^{\tau_h} - \bar{w}^{\tau_h} \|_{L^1(\Omega)} \leq \sqrt{d} \tau_h \TV(w)$ which implies $\int_{\Omega} w^{\tau_h}(x) \dvg \phi(x)\dd x$ $\le \TVh(w) + \eps(\tau_h,h,w)$.
	\end{proof}
	We are now able to prove the $\limsup$ inequality for $\TVh$.
	\begin{theorem}\label{thm:limsup_TV_TVh}
		Let $w \in L^1_W(\Omega)$ and tuples $\{(h,\tau_h)\}_{h>0}$ be given such that \cref{assu:grid} is fulfilled and $\frac{\tau_h}{h} \searrow 0$ as $h \searrow 0$. Then there holds
		\begin{align*}
			\TV(w) \geq \limsup_{h \searrow 0} \TVh( R^W_{P0^{\tau_h}}(w)).
		\end{align*}
	\end{theorem}
	\begin{proof}
		The claim follows from \cref{lem:estimate_TVhRP0hTV} by applying $\limsup_{h \searrow 0}$ to both sides of \eqref{eq:TVh_R_ineq} and using the assumption $\frac{\tau_h}{h} \searrow 0$.
	\end{proof}
	The embedding of the meshes $\QQ_{\tau_h}$ into the $\QQ_h$ and the superlinear coupling of their mesh sizes yield a lower bound estimate for the discretized total variation of the recovery sequence from \cref{thm:limsup_TV_TVh} using techniques from Proposition 3.7  in \cite{Caillaud2020Error}.
	\begin{proposition} \label{prop:TVh_lower_bound}
		Let $w \in \BVW(\Omega)$, $\TV(w) = - \int_\Omega \div \phi(x)  w(x) \dd x$ with $\phi \in W^{1,\infty}_0(\Omega;\R^d)$ and $\| \phi \|_{L^\infty(\Omega;\R^d)} \leq 1$, and let $\{(h,\tau_h)\}_{h >0}$ be coupled such that \cref{assu:grid} is fulfilled. If $\div \phi \in H^1(\Omega)$, then there holds
		\begin{align*}
			\TVh(R^W_{P0^{\tau_h}}(w)) \geq \TV(w) - c(w,\phi) \left(h^2 + h
				+ \frac{\tau_h}{h} \right)
		\end{align*}
		with some constant
		$c(w,\phi)>0$, depending
		on $w$ and $\phi$.
	\end{proposition}
	\begin{proof}
		Let $\phi^h \coloneqq I_{RT0^h} \phi \in RT0^h_0$ for $h >0$. By \eqref{eq:RT_linf_error_estimate}, there holds $\| \phi^h \|_{L^\infty(\Omega)} \leq 1 + C h$ with $C = C(\phi) >0$. We define $\tilde{\phi}^h \coloneqq \frac{1}{1+ Ch} \phi^h  \in RT0^h_0$ for $h >0$, which fulfills $\| \tilde{\phi}^h \|_{L^\infty(\Omega)} \leq 1$. Moreover, we define $w^h = \Pi_{P0^h}w$ for $h >0$, which fulfills $\int_\Omega w(x) w^h(x)\dd x = \| w^h \|_{L^2(\Omega)}^2$ due to the optimality condition of the projection. By \eqref{eq:interpolation_div}, there holds $\div I_{RT0^h} \phi = \Pi_{P0^h} \div \phi$ so that also $\int_\Omega \div \phi(x) \div \phi^h(x) \dd x = \| \div \phi^h \|_{L^2(\Omega)}^2$. Hence, there holds
		\begin{align*}
			&\TV(w) = - \frac{1}{2} \| \div \phi + w \|_{L^2(\Omega)}^2 + \frac{1}{2} \|w \|_{L^2(\Omega)}^2 + \frac{1}{2} \| \div \phi \|_{L^2(\Omega)}^2 \\
			& \leq - \frac{1}{2} \| \div \phi^h +w^h \|_{L^2(\Omega)}^2 + \frac{1}{2} \|w \|_{L^2(\Omega)}^2 + \frac{1}{2} \| \div \phi \|_{L^2(\Omega)}^2 \\
			& = \frac{1}{2} \big( \| \div \phi \|_{L^2(\Omega)}^2-\| \div \phi^h \|_{L^2(\Omega)}^2 + \|w \|_{L^2(\Omega)}^2 - \| w^h \|_{L^2(\Omega)}^2  \big) - \int_\Omega \div \phi^h(x) w^h(x) \dd x \\
			& = \frac{1}{2} \| \div (\phi^h - \phi) \|_{L^2(\Omega)}^2 + \frac{1}{2} \| w - w^h \|_{L^2(\Omega)}^2 - (1+Ch) \int_\Omega \div \tilde{\phi}^h(x) w^h(x) \dd x \\
			& \leq \frac{b^2 h^2}{2} \| \nabla \div \phi \|_{L^2(\Omega;\R^d)}^2  + \frac{1}{2} \| w -w^h \|_{L^1(\Omega)} \| w - w^h \|_{L^\infty(\Omega)} + (1+Ch) \TVh(w^h),
		\end{align*}
		where we used estimate \eqref{eq:est_norm_interpolation} to obtain the last inequality.
		\Cref{lem:estimate_ProjP0h} gives $\| w -w^h \|_{L^1(\Omega)} \leq \sqrt{d} h \TV(w)$. Since $w,w^h \in \BVW(\Omega)$, there holds $\| w-w^h \|_{L^\infty} \leq \max_{w_1,w_2 \in W} |w_1 - w_2 | \eqqcolon W_{\max}$. There holds $\TVh(w^h) = \TVh(\Pi_{P0^{\tau_h}}w)$ due to $\int_Q w^h(x) \dd x$ $= \int_Q \Pi_{P0^{\tau_h}}w(x) \dd x$ for all $Q \in \QQ_h$ by \cref{assu:grid}. By \cref{lem:TVh}, there exists $\hat{\phi}^h \in RT0^h_0$ with $\| \hat{\phi}^h \|_{L^\infty(\Omega)} \leq 1$ so that
		\begin{align*}
			\TVh(\Pi_{P0^{\tau_h}}w) \hspace{-1cm}& \hspace{1cm} = \int_\Omega \div \hat{\phi}^h(x) \Pi_{P0^{\tau_h}} w(x) \dd x \\
			&= \int_\Omega \div \hat{\phi}^h(x) ( \Pi_{P0^{\tau_h}} w(x) - R^W_{P0^{\tau_h}}(w)(x)) \dd x + \int_\Omega \div \hat{\phi}^h(x) R^W_{P0^{\tau_h}}(w) (x)\dd x \\
			& \leq \| \div \hat{\phi}^h \|_{L^\infty(\Omega)} \| \Pi_{P0^{\tau_h}}w - R^W_{P0^{\tau_h}}(w) \|_{L^1(\Omega)} + \TVh(R^W_{P0^{\tau_h}}(w)) \\
			& \leq \frac{4d \sqrt{d} \tau_h}{h} \TV(w) + \TVh(R^W_{P0^{\tau_h}}(w)),
		\end{align*}
		where the last inequality follows from \cref{lem:div_bound,lem:estimate_RWP0h}.
		In total, we obtain
		\begin{multline*}
			\TVh(R^W_{P0^{\tau_h}}(w)) - \TV(w)\\
			\ge - \frac{\left( Ch + \frac{1}{2}W_{\max} \sqrt{d} h + (1+Ch) \frac{4 d \sqrt{d} \tau_h}{h} \right) \TV(w) + \frac{b^2 h^2}{2} \| \nabla \div \phi \|_{L^2(\Omega;\R^d)}^2}{1+Ch}.
		\end{multline*}
	\end{proof}
	\section{Discretization of problem (P)} \label{sec:Gamma-convergence}
	We define the two optimization problems
	\begin{gather}\label{eq:P_c}
		\min_{w \in L^2(\Omega)}\ G(w) \coloneqq F(w) + \alpha \TV(w) + I_Z(w)
		\tag{\text{P$_c$}}
	\end{gather}
	and
	\begin{gather}\label{eq:P_c^h}
		\min_{(w,V) \in L^2(\Omega)\times \R}\ G^h(w,V) \coloneqq F(w) + \alpha V + I_{Z^h}(w,V)\tag{\text{P$^h_c$}}.
	\end{gather}
	The feasible sets are respectively given by $Z \coloneqq \left\{ w \in L^1_W(\Omega) \mid \TV(w) \leq c \TV(w) \right\}$
	and $Z^h \coloneqq \left\{ (w,V) \in (P0^{\tau_h} \cap L^1_W(\Omega)) \times \R \mid \TV(w) \leq c V, \, \TVh(w) \leq V \right\}$
	with a constant $c \geq 1$ and with $\tau_h > 0$ coupled to $h>0$ such that \cref{assu:grid} is fulfilled. The feasibility in \eqref{eq:P_c} and \eqref{eq:P_c^h} is ensured by the $\{0, \infty\}$-valued indicator functionals $I_Z$ and $I_{Z^h}$. The purpose of the constraint $\TV(w) \leq c V$ in \eqref{eq:P_c^h} is to obtain boundedness of the sequence of solutions to \eqref{eq:P_c^h} in $\BV(\Omega)$ which yields compactness in $\BV(\Omega)$. This also yields the existence of minimizers and will be analyzed in \Cref{sec:compactness}. Compared to \eqref{eq:p}, the problem \eqref{eq:P_c} has the additional constraint $\TV(w) \leq c \TV(w)$, which does not change the feasible set.
	\begin{lemma}
		\eqref{eq:P_c} is equivalent to \eqref{eq:p}.
	\end{lemma}
	\begin{proof}
		Because $c\geq1$ and $\TV(w) \geq 0$ for all $w \in L^1(\Omega)$, the constraint $\TV(w) \leq c \TV(w)$ is trivially fulfilled for all $w \in L^1(\Omega)$.
	\end{proof}
	Instead, the constraint $\TV(w) \leq c \TV(w)$ is used to illustrate that \eqref{eq:P_c} is the limit problem to the family of problems \eqref{eq:P_c^h} when driving $h \searrow 0$ in a $\Gamma$-convergence sense. Specifically, we obtain the following results.
	\begin{theorem}[Lim inf inequality]\label{thm:liminf}
		Let $F: L^1(\Omega) \to \R$ be continuous and bounded from below. Let $G$ and $G^h$ be defined as above with $c\geq1$. Consider tuples $\{(h,\tau_h)\}_{h >0}$ such that \cref{assu:grid} is fulfilled. Let $\{(w^{\tau_h},V^h)\}_{h >0} \subset L^1(\Omega) \times \R$ be a sequence with $w^{\tau_h} \to w$ in $L^1(\Omega)$ as $h \searrow 0$. Then there holds
		\begin{align*}
			G(w) \leq \liminf_{h \searrow 0 } G^h(w^{\tau_h},V^h).
		\end{align*}
	\end{theorem}
	\begin{proof}
		Without loss of generality, we may assume $(w^{\tau_h},V^h) \in Z^h$ for all $h > 0$ because $(w^{\tau_h},V^h) \notin Z^h$ implies the trivial case $G^h(w^{\tau_h},V^h) = \infty$. That is, in particular there holds $w^{\tau_h} \in L^1_W(\Omega)$ and $\TVh(w^{\tau_h}) \leq V^h$ for all $h >0$. By \cref{lem:convergence_in_L^p}, there holds $w \in L^1_W(\Omega)$. Together with \cref{thm:TV_liminf}, there holds $\alpha \TV(w) \leq \liminf_{h \searrow 0} \alpha \TVh(w^{\tau_h}) \leq \liminf_{h \searrow 0} \alpha V^h$.
		
		Again without loss of generality, we may assume that $\liminf_{h \searrow 0} \alpha V^h < \infty$, because that again would imply the trivial case $\liminf_{h \searrow 0} G^h(w^{\tau_h},V^h) = \infty$ since $F$ is bounded from below by assumption.
		Hence, we obtain by \cref{lem:convergence_in_L^p} that $w \in \BVW(\Omega)$ and therefore $w \in Z$, which implies $I_{Z^h}(w^{\tau_h}) = I_Z(w)$ for all $h >0$.
		Because $F: L^1(\Omega) \to \R$ is continuous, there holds $F(w) = \lim_{h \searrow 0} F(w^{\tau_h})$. In total, we obtain $G(w) \leq \liminf_{h \searrow 0} G^h(w^{\tau_h},V^h)$.
	\end{proof}
	\begin{theorem}[Lim sup inequality] \label{thm:limsup_inequality}
		Let $d \in \{1,2,3\}$ and $F: L^1(\Omega) \to \R$ be continuous and bounded from below. Let $c \geq 1$ in the case $d=1$, $c \geq \sqrt{2}$ in the case $d = 2$, and $c \geq 13 \sqrt{3}$ in the case $d=3$. Let tuples $\{(h,\tau_h)\}_{h >0}$ be given such that \cref{assu:grid} is fulfilled and $\frac{\tau_h}{h} \searrow 0$ as $h \searrow 0$. Let $w \in \BV_W(\Omega)$. Then there exists a sequence $\{(w^{\tau_h},V^h)\}_{h > 0}  \subset \BVW(\Omega) \times \R$ such that $(w^{\tau_h},V^h) \in Z^h$, $w^{\tau_h} \weakstarto w$ in $\BV(\Omega)$, and
		\begin{align*}
			G(w) \geq \limsup_{h \searrow 0} G^h(w^{\tau_h},V^h).
		\end{align*}
	\end{theorem}
	The proof of \cref{thm:limsup_inequality} is provided in \cref{sec:limsup}. In combination with the aforementioned compactness arguments, \cref{thm:liminf,thm:limsup_inequality} yield that minimizers of \eqref{eq:P_c^h} converge to a minimizer of \eqref{eq:P_c} when driving $h \searrow0$, which is the main result of this section.
	\begin{theorem}\label{thm:convergence_to_minimizer}
		Let $d \in \{1,2,3\}$ and $F: L^1(\Omega) \to \R$ be continuous and bounded from below. Let $c \geq 1$ for the case $d=1$, $c \geq \sqrt{2}$ for the case $d=2$, and $c \geq 13\sqrt{3}$ for the case $d=3$. Let tuples $\{(h,\tau_h)\}_{h >0}$ be coupled such that \Cref{assu:grid} is fulfilled and $\frac{\tau_h}{h} \searrow 0$ as $h \searrow 0$. Denote by $\{(w^{\tau_h},V^h)\}_{h >0} \subset \BVW(\Omega) \times \R$ with $w^{\tau_h} \in P0^{\tau_h}$ a sequence of optimal solutions to \eqref{eq:P_c^h}. Then $\{w^{\tau_h}\}_{h >0}$ admits a subsequence that converges weakly-$*$ in $\BV(\Omega)$ and each accumulation point of $\{w^{\tau_h} \}_{h>0}$ is a minimizer of \eqref{eq:P_c}.
	\end{theorem}
	Next we provide the compactness arguments and the proof of \cref{thm:convergence_to_minimizer}.
	\subsection{Existence of minimizers and compactness}\label{sec:compactness}
	In order to guarantee the existence of subsequences of sequences of optimal solutions to the problems \eqref{eq:P_c^h} that converge in $L^1_W(\Omega)$ as $h \searrow 0$, we have implemented compactness with the help of the constraint $\TVh(w) \leq c V$ in the sense of the following result, which follows from Theorem 3.23 in \cite{ambrosio2000functions} since $\Omega$ is a bounded Lipschitz domain.
	\begin{lemma} \label{lem:converging_subsequence}
		Let $\{u^h\}_{h >0} \subset \BV(\Omega)$ be bounded in $\BV(\Omega)$. Then $\{u^h\}_{h>0}$ admits a subsequence converging weakly-$*$ in $\BV(\Omega)$ to some $u \in \BV(\Omega)$.
	\end{lemma}
	Before continuing with the statements and proofs, we illustrate the effect of absence and presence of the constraints $\TV(w) \leq cV$ in \eqref{eq:P_c^h} as $h \searrow 0$ in an extension of \cref{expl:existence_of_minimizers}.
	\begin{remark}
		In \Cref{expl:existence_of_minimizers}, we stated that the replacement of $\TV$ by $\TVh$ in \eqref{eq:p} causes that the existence of minimizers is no longer guaranteed since the null space of $\TVh$ is greater than the null space of $\TV$.	
		If we additionally demand $w \in P0^{\tau_h}$
		we can now guarantee the existence of minimizers to the resulting problem because of its finite dimension. However, we now have the issue that the minimizers might not converge to a minimizer of \eqref{eq:p} when $h$ is driven to zero if we couple the mesh sizes $\tau_h$ and $h$ superlinearly as required for the recovery sequence according to \cref{thm:limsup_TV_TVh}. This coupling allows for chattering of the sequence of minimizers without affecting the discretized total variation on the coarser mesh.
		The constraint $\TV(w) \leq c V$ from \eqref{eq:P_c^h} prevents this chattering because it bounds the total variation of the minimizers since the function $F$ is bounded from below and hence yields weak-$*$ convergence in $\BV(\Omega)$ of a subsequence as stated in \cref{lem:converging_subsequence}.
	\end{remark}
	\begin{theorem}\label{thm:P^h_admits_solution}
		Let $F: L^1(\Omega) \to \R$ be continuous and bounded from below and let a tuple $(h,\tau_h)$ be given such that \Cref{assu:grid} is fulfilled. Then problem \eqref{eq:P_c^h} admits a solution.
	\end{theorem}
	\begin{proof}
		Problem \eqref{eq:P_c^h} has the same optimal value as the optimization problem
		\begin{gather}\label{eq:tilde{P}^h}
			\begin{aligned}
				\min_{w \in \BVW(\Omega)\cap P0^{\tau_h}}\ & F(w) + \alpha \max\left\{\frac{1}{c} \TV(w), \TVh(w) \right\}.
			\end{aligned}\tag{\text{$\mathrm{\tilde{P}^h}$}}
		\end{gather}
		Since the number of elements in $\BVW(\Omega) \cap P0^{\tau_h}$ is finite, problem \eqref{eq:tilde{P}^h} admits an optimal solution which we denote by $\bar{w}$. Define $\bar{V} \coloneqq \max \left\{ \frac{1}{c} \TV(\bar{w}), \TVh(\bar{w}) \right\}$, then $(\bar{w},\bar{V})$ is optimal for \eqref{eq:P_c^h}.
	\end{proof}
	\begin{remark}
		We could also prove the existence of minimizers of \eqref{eq:P_c^h} without the discretization of $w$, that is, with $Z^h \coloneqq \left\{ w \in L^1_W(\Omega) \mid  \TV(w) \leq c V, \, \TVh(w) \leq V \right\}$ because the constraint $\TV(w) \leq cV$ bounds minimizing sequences in $\BV(\Omega)$. This yields the existence of a subsequence that converges weakly-$*$ in $\BV(\Omega)$ by \cref{lem:converging_subsequence}.
	\end{remark}
	\begin{lemma} \label{lem:minimizer_bounded}
		Let $\{(h,\tau_h)\}_{h >0}$ be coupled such that \cref{assu:grid} is fulfilled and let $(w^{\tau_h},V^h)$ denote an optimal solution to \eqref{eq:P_c^h}. Then $\{(w^{\tau_h},V^h)\}_{h > 0}$ is bounded in $\BV(\Omega) \times \R$.
	\end{lemma}
	\begin{proof}
		The functional $F$ is bounded from below, that is, there exists a constant $\ell \in \R$ such that $F(w) \geq \ell$ for all $w \in L^1(\Omega)$. Define $\tilde{w} \equiv w_1$ for some $w_1 \in W$. Then $(\tilde{w},0)$ is feasible for \eqref{eq:P_c^h} for all $h >0$ with objective value $G^h(\tilde{w},0) = F(\tilde{w})$ since $\TVh(\tilde{w}) = \TV(\tilde{w})= 0$. Now let $(w^{\tau_h},V^h)$ be optimal for \eqref{eq:P_c^h} for $h >0$. Then $F(w^{\tau_h}) + \alpha V^h \leq F(\tilde{w})$ and therefore $V^h \leq \frac{1}{\alpha}(F(\tilde{w})-\ell)$ for all $h>0$. This yields that $\TV(w^{\tau_h}) \leq c V^h \leq \frac{c}{\alpha}(F(\tilde{w})-\ell)$.
		The boundedness in $L^1(\Omega)$ follows directly from the boundedness in $L^\infty(\Omega)$, which is due to $w^{\tau_h}(x) \in W$ for a.a.\ $x \in \Omega$.
	\end{proof}
	\begin{proof}[Proof of \cref{thm:convergence_to_minimizer}]
		By \Cref{lem:minimizer_bounded,lem:converging_subsequence}, the sequence $\{w^{\tau_h}\}_{h >0}$ admits a subsequence that converges weakly-$*$ in $\BV(\Omega)$ to some $w \in \BV(\Omega)$ which we denote by the same symbol, that is, $w^{\tau_h} \weakstarto w$ in $\BV(\Omega)$. By \Cref{thm:liminf}, there holds $G(w) \leq \liminf_{h \searrow 0} G^h(w^{\tau_h},V^h)$.
		By \Cref{thm:limsup_inequality}, there exists for each $(v,U) \in \BVW(\Omega) \times \R$ a sequence $\{(v^h,U^h)\}_{h >0}$ $\subset \BVW(\Omega) \times \R$ such that $v^h \in P0^{\tau_h}$, $v^h \weakstarto v$ in $\BV(\Omega)$, and $G(v) \geq \limsup_{h \searrow 0} G^h(v^h,U^h)$.
		Together we then have
		\begin{align*}
			G(w) \leq \liminf_{h \searrow 0} G^h(w^{\tau_h},V^h) \leq \liminf_{h \searrow 0} G^h(v^h,U^h) \leq \limsup_{h \searrow 0} G^h(v^h,U^h) \leq G(v)
		\end{align*}
		for each $v \in \BVW(\Omega)$. This yields the optimality of $w$ for \eqref{eq:P_c}.
	\end{proof}
	\begin{remark}
		\Cref{thm:convergence_to_minimizer} implicitly yields the existence of a minimizer of \eqref{eq:P_c} and therefore for \eqref{eq:p}.
	\end{remark}
	\subsection{Lim sup inequality} \label{sec:limsup}
	In this section, we prove the $\limsup$ inequality for the functions $G$ and $G^h$,
	which is stated in \Cref{thm:limsup_inequality}.
	We need to find a constant $c$ so that the inequality in the
	set $Z^h$ can be satisfied by a recovery sequence.
	The admissible range of the constant $c$ depends on the dimension $d$ so that we provide different proofs depending on $d$.
	As an intermediate step, we first prove the inequality
	$\TV(R^W_{P0^\tau}(w)) \leq c \TV(w)$
	in each case for a respective constant $c \geq 1$ depending on $d$.
	For the case $d=1$ it is immediate that any constant $c \geq 1$ is permissible.
	Besides the theoretical requirements on the constant $c$, its choice may have a major impact on the numerical results. For a fixed mesh size, a larger constant $c$ allows for more chattering of the solution within the coarser mesh cells.
	On the other hand, a smaller constant $c$ leads to a stronger upper bound on the total variation of approximate numerical solutions and therefore less chattering. This will be confirmed
	by our computational results in \cref{sec:Numeric}.
	\begin{lemma}
		Let $d=1$ and $\QQ_\tau$ be a partition of $\Omega$ into intervals of length $\tau >0$. Let $w \in \BVW(\Omega)$. There holds $\TV(R^W_{P0^{\tau}}(w)) \leq \TV(w)$.
	\end{lemma}
	\begin{proof}
		Let $\Omega = (a_\ell,a_u)$ with $\QQ_\tau = \{Q_1, \dots, Q_n\}$ with $Q_i = (a_i,a_{i+1}]$ for $i = 1,\dots, n-1$ and $Q_n = (a_n,a_{n+1})$ with $a_\ell = a_1 < a_2 < \dots < a_n < a_{n+1} = a_u$.  Let $R^W_{P0^\tau}|_{Q_i} \equiv w_i$ with $w_i \in W$ for $i=1,\dots,n$. \Cref{dfn:RWP0tau} yields that $\left| w|_{Q_i}^{-1}(w_i) \right| >0$ for each $i = 1,\dots,n$, so that
		$\TV(R^W_{P0^\tau}(w)) = \sum_{i=1}^n |w_i - w_{i+1}| \leq \TV(w)$.
	\end{proof}
	For the cases $d=2$ and $d=3$, we provide a constant that is not sharp but can be proven straightforwardly in \Cref{thm:TV_estimate_d23}. In \Cref{thm:tvh_tv_d=2}, we will improve upon this in the case $d=2$ and prove that $\sqrt{2}$ is a sharp lower bound on $c$ in that case. Since the proof of \Cref{thm:tvh_tv_d=2} needs technical arguments, we only provide a sketch of the proof and
	postpone the detailed proof and the preparatory results to \cref{appendix:A}.
	\begin{theorem} \label{thm:TV_estimate_d23}
		Let $d \in \{2,3\}$ and $\QQ_\tau$ be a partition of $\Omega$ into axis-aligned squares or cubes $Q \in \QQ_\tau$ of height $\tau >0$. Let $w \in \BVW(\Omega)$. There holds
		\begin{align*}
			\TV(R^W_{P0^\tau }(w))  \leq c_d \TV(w)
		\end{align*}
		with $c_d \coloneqq (4d+1) \sqrt{d}$.
	\end{theorem}
	\begin{proof}
		We construct a function $\phi \in H_0(\div;\Omega)$ to realize the exact value of $\TV(R^W_{P0^\tau }(w))$ following the construction of Raviart--Thomas basis functions in \cite{bahriawati2005RT}. Afterwards, we apply
		\cref{lem:TV_div} and use $\phi$ as a test function for $\TV(w)$ to estimate $\TV(R^W_{P0^\tau }(w))$ against $\TV(w)$. We denote the collection of facets in the interior of $\Omega$ by $\EE$ and $\bar{w} \coloneqq R^W_{P0^\tau}(w)$. Consider a facet $E \in \EE$. Then $E = \overline{Q^E_1} \cap \overline{Q^E_2}$ for $Q_1^E,Q_2^E \in \QQ_{\tau}$ with $Q_1^E \neq Q_2^E$. Denote by $q_1^E, q_2^E \in \R^d$ the respective centers of the cubes $Q_1^E$ and $Q_2^E$ and define $T_i^E \coloneqq \conv(E \cup \{q_i^E\})$ for $i \in \{1,2\}$. We define the function $\phi_E$ by
		\begin{align*}
			\phi_E(x) = \begin{cases}
				(-1)^i s_E \frac{2}{\tau} (x - q_i^E) & \text{for } x \in T_i^E, \, i \in \{1,2\} \\
				0 & \text{elsewhere,}
			\end{cases}
		\end{align*}
		where $s_E \coloneqq \mathrm{sgn} (\bar{w} |_{Q_2^E} - \bar{w} |_{Q_1^E}) \in \{-1,0,1\}$. There holds that $\phi_E \in H_0(\div; \Omega)$ because $\div \phi_E = (-1)^i d s_E\frac{2}{\tau}$ on $T_i^E , \, i \in \{1,2\}$, and $\div \phi_E = 0$ elsewhere as well as $\phi_E \cdot n_E = s_E$ on $E$,
		$\phi_E \cdot n_{T_i^E} = 0$ on $\partial T_i^E \setminus E$ as in \cite{bahriawati2005RT}, where $n_{T_i^E}$ denotes the outer normal to $T_i^E$, $i \in \{1,2\}$,
		and $\phi_E \cdot n_E =  0$ on $H \in \EE \setminus E$,
		where $n_E$ denotes the unit normal vector of $E$ pointing from $Q_2^E$ to $Q_1^E$. Moreover, there holds $\| \phi_E \|_{L^\infty(\Omega;\R^d)}\leq \sqrt{d}$.
		Note that the supports of all $\phi_E$ only intersect on a set of Lebesgue measure zero. We define $\phi \coloneqq \frac{1}{\sqrt{d}} \sum_{E \in \EE} \phi_E$, then $\phi \in H_0(\div;\Omega)$ with $\| \phi \|_{L^\infty(\Omega;\R^d)} \leq 1$. There holds
		\begin{align*}
			\int_\Omega \bar{w} (x)\div \phi (x) \dd x & = \frac{1}{\sqrt{d}} \sum_{E \in \EE} \int_E ( \bar{w} |_{Q_2^E} - \bar{w} |_{Q_1^E}) \, \phi_E (x) \cdot n_E (x) \dd \Ha^{d-1}(x) \\
			& = \frac{1}{\sqrt{d}} \sum_{E \in \EE} \int_E | \bar{w} |_{Q_2^E} - \bar{w} |_{Q_1^E} | \dd \Ha^{d-1}(x) = \frac{1}{\sqrt{d}} \TV(\bar{w}),
		\end{align*}
		where the last equality follows from \cite[Lemma 2.1]{manns2023integer}, and therefore
		\begin{align*}
			\TV(\bar{w})
			& = \sqrt{d} \int_\Omega (\bar{w}(x) - w(x)) \div \phi(x) \dd x + \sqrt{d} \int_\Omega w(x) \div \phi(x) \dd x \\
			& \leq \sqrt{d} \| \bar{w} - w \|_{L^1(\Omega)} \| \div \phi \|_{L^\infty(\Omega)}+ \sqrt{d} \TV(w)  \leq (4d+1) \sqrt{d}\TV(w),
		\end{align*}
		where we have used \cref{lem:estimate_RWP0h} for the second inequality.
	\end{proof}
	For the case $d=2$, we improve the constant $c$.
	\begin{theorem}\label{thm:tvh_tv_d=2}
		Let $d = 2$ and $F: L^1(\Omega) \to \R$ be continuous and bounded. Consider tuples $\{(h,\tau_h)\}_{h >0}$ such that \Cref{assu:grid} is fulfilled and $\frac{\tau_h}{h} \searrow 0$ as $h \searrow 0$. Let $w \in \BVW(\Omega)$ be given. There is a sequence $\{w^h\}_{h >0}$ with $w^h \in P0^{\tau_h} \cap \BVW(\Omega)$ as $h \searrow 0$ such that $w^h \weakstarto w$ in $\BV(\Omega)$,
		\begin{align*}
			\limsup_{h \searrow 0} \TV(w^h) \leq \sqrt{2} \TV(w), \quad \text{and} \quad \limsup_{h \searrow 0} \TVh(w^h) \leq \TV(w).
		\end{align*}
	\end{theorem}
	\begin{proof}[Sketch of proof]
		The proof is based on a diagonal sequence selection argument that
		chains several approximations.
		First, we approximate $w$ strictly by functions with polygonal jump sets by means of \Cref{lem:approx_polygonal}. Then we apply the rounding operator  $R^W_{P0^{\tau_h}}$ to the elements of the sequence to obtain functions in $P0^{\tau_h}$ with values in $W$. Next, we establish \cref{thm:TV_polygonal} to estimate the total variation of a rounding of a function with polygonal jump sets against the total variation of its preimage for vanishing grid size $\tau_h$, which yields the constant $\sqrt{2}$. From this, a diagonal sequence is selected by balancing the approximation by functions with polygonal jump sets with the approximation of the rounding operator.
	\end{proof}
	The following example illustrates that the constant $\sqrt{2}$ in \Cref{thm:tvh_tv_d=2} is sharp, that is, there cannot exist $c < \sqrt{2}$
		that satisfies $\limsup_{h\searrow 0} \TV(R^W_{P0^{\tau_h}}(w)) \leq c \TV(w)$ uniformly for all $w \in \BVW(\Omega)$
		if $d = 2$.
	\begin{example} \label{expl:constant_sharp}
		Let $d=2$, $\Omega = (0,1)^2$, $W= \{0,1\}$, and $w \in \BVW(\Omega)$ be defined by $w= \chi_{\{(-1,1)^Tx > 0\}}$.
		Then $\TV(w) = \sqrt{2}$. Consider the family of meshes $\{\QQ_{\tau_k}\}_{k \in \N}$ with $\tau_k = \frac{1}{k}$ for $k \in \N$. Then $\TV(R^W_{P0^{\tau_k}}(w)) = \frac{2(k-1)}{k} \to 2$ as $k \to \infty$. This yields that the constant $\sqrt{2}$ in \Cref{thm:TV_polygonal} is sharp.
	\end{example}
	\begin{remark}
		The constant $\sqrt{2}$ in \cref{thm:tvh_tv_d=2} is sharp for the case $d=2$ as we proved in \cref{expl:constant_sharp}. We are convinced that a similar result holds for $d=3$ that can be obtained with a similar proof strategy. The necessary technical effort goes beyond the scope of this work, however.
	\end{remark}
	With the help of the former statements, we can now prove \Cref{thm:limsup_inequality}, which is stated in the beginning of the section.
	\begin{proof}[Proof of \Cref{thm:limsup_inequality}]
		To prove the $\limsup$ inequality, we may assume $w \in Z$ and $\TV(w) < \infty$ because otherwise the inequality follows immediately because in this case there holds $G(w) = \infty$. We consider the components of the objective functions $G$ and $G^h$ separately and start by proving $\limsup_{h \searrow 0} \alpha V^h \leq \alpha \TV(w)$. For the case $d =2$, we consider the sequence $\{w^{\tau_h}\}_{h >0}$ from \cref{thm:tvh_tv_d=2} and for the cases $d=1$ and $d=3$ we choose $w^{\tau_h} \coloneqq R^W_{P0^{\tau_h}}(w)$ for $h>0$. Then there holds
		\begin{align}
			\limsup_{h \searrow 0} \TV(w^{\tau_h}) \leq c \TV(w) \label{eq:limsup_const}
		\end{align}
		for the respective $c \geq 1$ in the case $d=1$, $c \geq \sqrt{2}$ in the case $d=2$, and $c \geq 13 \sqrt{3}$ in the case $d=3$. We define $\overline{T} \coloneqq \limsup_{h \searrow 0} \TV(w^{\tau_h}) \leq c \TV(w) < \infty$ and $\underline{T} \coloneqq \liminf_{h \searrow 0} \TVh(w^{\tau_h}) \geq 0$. Then we define the following two sequences $\delta^h_1 \coloneqq \max \{0,\TV(w^{\tau_h}) - \overline{T}\}$ and $\delta^h_2 \coloneqq \max \{0, \underline{T} - \TVh(w^{\tau_h})\}$, which fulfill $\frac{\delta^h_1}{c} + \delta^h_2 \to 0$ as $h \searrow 0$. We define $V^h \coloneqq \TVh(w^{\tau_h}) + \frac{\delta^h_1}{c} + \delta^h_2$. There holds
		\begin{align*}
			\limsup_{h \searrow 0} \TVh(w^{\tau_h}) + \frac{\delta^h_1}{c} + \delta^h_2 \leq \limsup_{h \searrow 0} \TVh(w^{\tau_h}) + \lim_{h \searrow 0} \frac{\delta^h_1}{c} + \delta^h_2 \leq \TV(w)
		\end{align*}
		by \cref{thm:tvh_tv_d=2} and therefore $\limsup_{h \searrow 0} \alpha V^h \leq \alpha \TV(w)$.
		Next, we consider the indicator functionals $I_{Z^h}$ and $I_Z$. Since $\TV(w) \leq c \TV(w)$ is trivially fulfilled due to $c \geq 1$, there holds $I_Z(w) = 0$. We prove that $I_{Z^h}(w^{\tau_h},V^h) = 0$ for all $h >0$, that is we need to prove that $\TVh(w^{\tau_h}) \leq V^h$ and $\TV(w^{\tau_h}) \leq c V^h$ for all $h >0$. The inequality $\TVh(w^{\tau_h}) \leq V^h = \TVh(w^{\tau_h}) + \frac{\delta^h_1}{c} + \delta^h_2$ is trivially fulfilled due to $\frac{\delta^h_1}{c}, \delta^h_2 \geq 0$. For the second inequality, there holds
		\begin{align*}
			\TV(w^{\tau_h}) - \delta^h_1 & \leq \limsup_{h \searrow 0} \TV(w^{\tau_h}) \leq c \TV(w) \\
			& \leq c \liminf_{h \searrow 0} \TVh(w^{\tau_h}) \leq c (\TVh(w^{\tau_h})+ \delta^h_2),
		\end{align*}
		where the second inequality is due to \eqref{eq:limsup_const} and the third follows from \Cref{thm:TV_liminf}, which yields $\TV(w^{\tau_h}) \leq c V^h$.
		
		Finally, since the function $F:L^1(\Omega) \to \R$ is continuous, there holds $F(w) = \lim_{h \searrow 0} F(w^{\tau_h})$ because $w^{\tau_h} \to w$ in $L^1(\Omega)$.
		In total, we have
		\begin{align*}
			G(w) & = F(w) + \alpha \TV(w) + I_Z(w) \geq \lim_{h \searrow 0} F(w^{\tau_h}) + \limsup_{h \searrow 0} \alpha V^h + \lim_{h \searrow 0} I_{Z^h}(w^{\tau_h},V^h) \\
			& \geq \limsup_{h \searrow 0} F(w^{\tau_h}) + \alpha V^h + I_{Z^h}(w^{\tau_h},V^h) = \limsup_{h \searrow 0} G^h(w^{\tau_h},V^h).
		\end{align*}
	\end{proof}
	\section{Outer Approximation} \label{sec:OA}
	In this section, we provide an outer approximation algorithm for solving \eqref{eq:P_c^h}. The inequality $\TVh(w) \leq V$ in \eqref{eq:P_c^h} is equivalent to
	\begin{align}
		\int_{\Omega} \div \phi (x) w (x) \dd x \leq V \quad \forall \, \phi \in RT0^h_0 \text{ with } \| \phi \, \|_{L^\infty(\Omega;\R^d)} \leq 1. \label{eq:cutting_planes}
	\end{align}
	The idea of the algorithm is to start by solving \eqref{eq:P_c^h} without the constraint $\TVh(w) \leq V$ and iteratively add the inequalities in \eqref{eq:cutting_planes} to cut off infeasible solutions until the remaining violation of \eqref{eq:cutting_planes} vanishes. The algorithm is stated in \cref{alg:outer_approx}.
	\begin{remark}
		The idea of using a cutting plane strategy for approximating the total variation seminorm
		is in no way limited to the integrality restriction we impose. A further article that
		analyzes this subject in convex control settings is currently being prepared by Meyer and Schiemann
		\cite{ms2024tv}.
	\end{remark}
	\begin{algorithm}[t]
		\caption{Outer approximation algorithm for \eqref{eq:P_c^h}}\label{alg:outer_approx}
		\textbf{Input:} $F$ sufficiently regular, $\alpha > 0$, $\tau >0$, $h >0$.
		\begin{algorithmic}[1]
			\State Set $k = 0$.
			\State Compute an optimal solution $(w^k,V^k)$ to
			\begin{gather*}\label{eq:MIP}
				\begin{aligned}
					\min_{(w,V)} \enskip & F(w) + \alpha V \\
					\mathrm{s.t.} \enskip\, & \TV(w) \leq cV  \\
					& \int_\Omega \div \phi^i (x) \, w (x) \dd x \leq V \quad \forall \, i \in \N, i \leq k \\
					& w \in \BVW(\Omega) \cap P0^\tau.
				\end{aligned}\tag{MIP}
			\end{gather*}
			\State Compute $\TVh(w^k)$ by solving
			\begin{gather*} \label{eq:QP}
				\begin{aligned}
					\max_{\phi} \enskip & \int_{\Omega} \div \phi (x) \, w^k(x) \dd x \quad
					\mathrm{s.t.} \quad \phi \in RT0^h_0, \enskip \| \phi \|_{L^\infty(\Omega;\R^d)} \leq 1
				\end{aligned} \tag{QP}
			\end{gather*}
			with optimal solution $\phi^{k+1}$.
			\If{$\int_\Omega \div \phi^{k+1}(x) \, w^k(x) \dd x - V^k \leq 0$}
			\State \Return $w^k$ as optimal solution
			\EndIf
			\State Set $k = k+1$ and go to Step 2.
		\end{algorithmic}
	\end{algorithm}

	\begin{theorem}
		\Cref{alg:outer_approx} is well-defined, stops after finitely many iterations, and returns an optimal solution to \eqref{eq:P_c^h}.
	\end{theorem}
	\begin{proof}
		\Cref{alg:outer_approx} is well-defined because problem \eqref{eq:MIP} admits an optimal solution due to its finite and non-empty feasible set and the boundedness of the objective function from below and problem \eqref{eq:QP} admits an optimal solution by \cref{lem:TVh}.
		
		Assume that the algorithm finds $(w^i,V^i)$ in step 2 in iteration $i$ and $(w^j,V^j)$ in step 2 of iteration $j$ with $i < j$ and $w^i = w^j$. Since $\TVh(w^j) = \TVh(w^i) = \int \div \phi^{i+1} w^i \dd x \leq V^j$ and $\TV(w^j) = \TV(w^i) \leq c V^j$, the algorithm terminates and $w^j=w^i$ is the optimal solution to \eqref{eq:P_c^h}. This yields that the algorithm always stops after finitely many iterations since the number of functions in $\BVW(\Omega) \cap P0^\tau$ is finite.
	\end{proof}
	\section{Numerical experiments} \label{sec:Numeric}
	\begin{figure}
		\centering
		\includegraphics[scale=0.5]{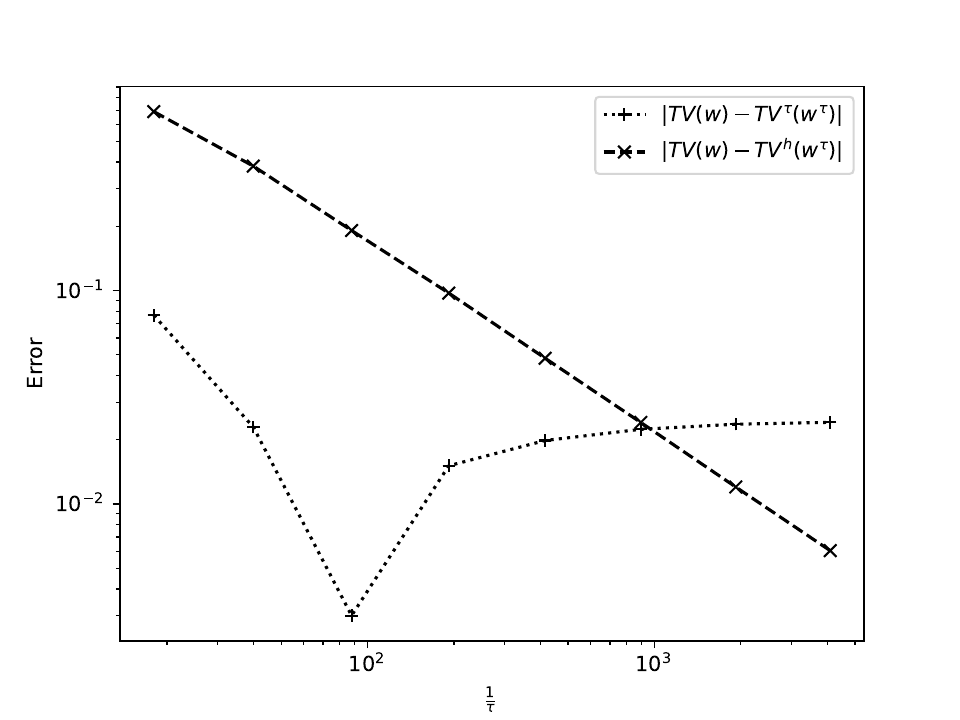}
		\caption{Errors $| \TV(w) - {\TV}^\tau(w^\tau)|$ and $|\TV(w) - \TVh(w^\tau) |$ for the pairs $(h,\tau)$ listed in \cref{fig:table_TV} with $w^\tau = R^W_{P0^\tau}(w)$.
		}\label{fig:TVh_TV}
	\end{figure}
	To demonstrate our theoretical results, we carried out two numerical experiments
	that were run on a single node of the Linux HPC cluster LiDO3 with two AMD EPYC 7542 32-Core CPUs and 64 GB RAM (computations were restricted to one CPU). We used DOLFINx 0.7.2 \cite{Dolfinx} for the finite element discretization and Gurobi 10.0.3 \cite{Gurobi} to solve the occurring optimization problems \eqref{eq:MIP} and \eqref{eq:QP}.
	
	The first numerical experiment is concerned with the approximation of $\TV$ with $\TVh$. \Cref{expl:TVh} demonstrated that there are functions $w \in \BVW(\Omega)$ such that $\TV(w)$ can not be approximated with ${\TV}^\tau(R^W_{P0^\tau}(w))$. We solved this issue by discretizing $w$ on a finer mesh than the mesh for the total variation and proved in \cref{thm:TV_liminf,thm:limsup_TV_TVh} that we can approximate $\TV(w)$ with $\TVh(R^W_{P0^{\tau}}(w))$ when the meshes $\QQ_h$ and $\QQ_{\tau}$ are coupled such that $\frac{\tau}{h} \searrow 0$ as $h \searrow 0$ and \cref{assu:grid} is fulfilled.
	We take up \cref{expl:TVh} and consider the function $w = \chi_{\{(\frac{1}{3},-1)^Tx \geq 0\}} \in \BVW(\Omega)$ with $\Omega = (0,1)^2$ and $W=\{0,1\}$. We discretized $\Omega$ into meshes $\QQ_h$ and $\QQ_\tau$ of axis-aligned squares for decreasing $h$ and $\tau$ fulfilling \cref{assu:grid}, see \cref{fig:table_TV}. We computed the values ${\TV}^\tau(w^\tau)$ and $\TVh(w^\tau)$ with $w^\tau \coloneqq R^W_{P0^\tau}(w)$. As predicted, we
	perceive that ${\TV}^\tau(w^\tau) \to \frac{1+\sqrt{5}}{3} \approx 1.07869 > \TV(w)$ as $\tau \searrow 0$ and $\TVh(w^\tau) \to \TV(w) = \frac{\sqrt{10}}{3} \approx 1.05409$ as $h \searrow 0$.
	Moreover, we observe an experimental error $\TV(w) - \TVh(w^\tau) \leq r \, \frac{\tau}{h}$ with $r \approx 8.513 \cdot 10^{-3}$, see \cref{fig:table_TV},
	in accordance with our theoretical error estimates in 
	\cref{prop:TVh_lower_bound,lem:estimate_TVhRP0hTV}.
	The non-vanishing error $|\TV(w) - {\TV}^\tau(w^\tau)|$
		and the vanishing error $|\TV(w) - \TVh(w^\tau)|$ are plotted for $\frac{\tau}{h}
		\searrow 0$, $h \searrow 0$ in \cref{fig:TVh_TV}.
	Another advantage of using two varying meshes for the input functions and the total variation besides the approximation is that the computation of $\TVh(w^\tau)$ is faster than the computation of ${\TV}^\tau(w^\tau)$ for the same $\tau >0$, which is due to the smaller size of the optimization problem one needs to solve to compute $\TVh(w^\tau)$ compared to ${\TV}^\tau(w^\tau)$. For $h^{-1} = 256$ and $\tau^{-1} = 256 \cdot 16$, we needed $831.16$ seconds to compute $\TVh(w^\tau)$ and $30933.33$ seconds to compute ${\TV}^\tau(w^\tau)$.
	
	In \cref{expl:constant_sharp}, we provided an example in which the constant $c = \sqrt{2}$ is sharp for the case $d=2$. Nevertheless, there are instances in which the constant could be chosen within the interval $[1,\sqrt{2})$. For example, if the limit function is a square whose sides lie on the grid lines of the mesh, then even the choice $c =1$ is valid. In our first numerical example whose results are stated in \cref{fig:table_TV}, we can estimate the asymptotically valid constant experimentally by considering the limit $\lim_{\frac{\tau}{h} \searrow 0} \frac{\TV(w^\tau)}{\TVh(w^\tau)} \approx 1.27$. In general, the choice $c \in [1,\sqrt{2})$ can only be validated if the solution to the limit problem is known while the choice $c \geq \sqrt{2}$ is always valid by \cref{thm:tvh_tv_d=2}.
	Due to this certainty, we believe that $c \ge \sqrt{2}$ is generally the preferred choice in optimization contexts, where the solution to the limit problem is unknown.
	
	As a second numerical example, we consider \eqref{eq:p} with the choices $F(w) = \| w - w_d \|_{L^1(\Omega)}$ and $\alpha = 5 \cdot 10^{-3}$ , that is, we consider the imaging optimization problem
	\begin{gather*}\label{eq:P_I}
		\begin{aligned}
			\min_{w\in L^2(\Omega)}\ & \| w -w_d \|_{L^1(\Omega)} + 0.005 \TV(w)\\
			\text{s.t.}\quad & w(x) \in W \subset \Z
			\text{ for a.a.\ } x \in \Omega,
		\end{aligned}\tag{\text{P$_I$}}
	\end{gather*}
	where $w_d \in L^1(\Omega)$ represents a noisy image. To keep the problem computationally manageable with an off-the-shelf solver, we decided to measure the distance between $w$ and $w_d$ in the $L^1$-norm instead of the more common $L^2$-norm because the $L^1$-norm yields mixed-integer linear programs after discretization. This is also the problem class that arises as subproblems in \cite{manns2023integer}, but since the subproblems have additional constraints which make them computationally more expensive, we leave them as future work. The original picture is shown in \cref{fig:cats} (a) and the noisy version is shown in \cref{fig:cats} (b). To obtain $w_d$, we added Gaussian noise with standard deviation $\sigma = 5 \%$ to the original picture and scaled the gray scale values to the interval $[0,5]$. We set $W = \{ 0, \dots, 5\}$ which represents six evenly distributed gray scale values.
	
	For the discretization \eqref{eq:P_c^h} of the problem above, we chose the mesh sizes $(h,\tau) = (\frac{1}{32}, \frac{1}{128})$ and $(h,\tau) = (\frac{1}{64},\frac{1}{512})$ and the constants $c = 3^i \sqrt{2}$ with $i=0,\dots,4$. We applied \cref{alg:outer_approx} with an iteration limit of $25$ iterations and a tolerance of $\num{e-03}$ for the gap $\frac{\TVh(w)-V}{\TVh(w)}$. We set the time limit for Gurobi to solve the mixed-integer linear programs \eqref{eq:MIP} to $48$ hours and the acceptable optimality gap to the default value $\num{e-04}$. We present the results for the case $(h,\tau) = (\frac{1}{32}, \frac{1}{128})$ in \cref{fig:table_cats_32_4} and for the case $(h,\tau) = (\frac{1}{64},\frac{1}{512})$ in \cref{fig:table_cats_64_8}. In the tables, we provide the respective value of the constant $c$, the information why the algorithm terminated, the number of iterations until termination, the objective value, the values $\TV(w)$, $\TVh(w)$, $V$, the gap $\frac{\TVh(w)-V}{\TVh(w)}$ of the last iterate $(w,V)$, and the running time in seconds. In column $2$ of \cref{fig:table_cats_32_4,fig:table_cats_64_8}, the abbreviations indicate the reason for the termination of the algorithm, where \emph{Opt} means that we have found an optimal solution, \emph{Tol} means that that $\frac{\TVh(w)-V}{\TVh(w)} \leq 10^{-3}$, \emph{MaxIter} means that \cref{alg:outer_approx} reached the iteration maximum, and \emph{GrbTime} means that Gurobi reached the time limit while solving \eqref{eq:MIP}. The resulting images for the case $(h,\tau) = (\frac{1}{64},\frac{1}{512})$ as well as the original and the noisy image are shown in \cref{fig:cats}.
	
	In our numerical experiments, we observed that \cref{alg:outer_approx} was mostly able to close the gap between $V$ and $\TVh(w)$ within the first few iterations to an accuracy of order $\num{e-02}$ but it was not able to close the gap to the desired accuracy of $\num{e-03}$ within the prescribed iteration limit of $25$ iterations. Further experiments suggest that a moderate increase of the iteration limit is not sufficient to achieve an accuracy of $\num{e-03}$ in this example.
	
	We highlight the impact of the choice of the constant $c$ in \cref{fig:cats} for the discretized problems \eqref{eq:P_c^h}, even if the specific choice of $c \geq \sqrt{2}$ does not make a difference in the limit from a theoretical point of view. In \cref{fig:cats} (c), we chose the constant $c = \sqrt{2}$ from \cref{thm:tvh_tv_d=2} and in \cref{fig:cats} (e) the constant $c = 9 \sqrt{2}$ from \cref{thm:TV_estimate_d23} which leads to significantly different results. In particular, the constraint $\TV(w) \leq c V$ for smaller $c$ filters out chattering as we obtained for $c = 9 \sqrt{2}$ and $c=27 \sqrt{2}$ in \cref{fig:cats} (e) and (f). This might also avoid effects as observed in Figure 11 in \cite{chambolle2021approximating}, where the authors detected diffuse solutions to an inpainting problem discretized with the Raviart--Thomas approach from \cite{Caillaud2020Error}.
	\begin{table}[h]
		\centering
		\caption{Values of $h^{-1}$, $\tau^{-1}$, ${\TV}^\tau(w^\tau)$, $\TVh(w^\tau)$, and $\TV(w^\tau)$ with $w^\tau = R^W_{P0^\tau}(w)$.}	
		\resizebox{\textwidth}{!}{
			\begin{tabular}{cccccccc}
				\toprule
				$h^{-1}$ & $\tau^{-1}$ & ${\TV}^\tau(w^\tau)$ & $\TVh(w^\tau)$ & $\TV(w^\tau)$ & $\TV(w)$ 
				& $\frac{\TV(w^\tau)}{\TV^h(w^\tau)}$
				& $(\TV(w) - \TVh(w^\tau))\frac{h}{\tau}$ \\
				\midrule
				2 & $2 \cdot 9$       & 0.97748 & 0.36456 & 1.22222 & 1.05409 & 3.35259 & \num{8.513e-03} \\
				4 & $4 \cdot 10$      & 1.03118 & 0.67116 & 1.27500 & 1.05409 & 1.89970 & \num{2.292e-03} \\
				8 & $8 \cdot 11$      & 1.05709 & 0.86291 & 1.30682 & 1.05409 & 1.51443 & \num{2.727e-04} \\
				16 & $16 \cdot 12$    & 1.0692  & 0.95678 & 1.32292 & 1.05409 & 1.38268 & \num{1.259e-03} \\
				32 & $32 \cdot 13$    & 1.07393 & 1.00593 & 1.32692 & 1.05409 & 1.31910 & \num{1.526e-03} \\
				64 & $64 \cdot 14$    & 1.07648 & 1.02999 & 1.33036 & 1.05409 & 1.29162 & \num{1.599e-03} \\
				128 & $128 \cdot 15$  & 1.07774 & 1.04208 & 1.33229 & 1.05409 & 1.27849 & \num{1.577e-03} \\
				256 &  $256 \cdot 16$ & 1.07822 & 1.04805 & 1.33276 & 1.05409 & 1.27166 & \num{1.508e-03} \\
				\bottomrule
		\end{tabular}}
		\label{fig:table_TV}
	\end{table}
	\begin{figure}[h]
		\centering
		\begin{minipage}{\textwidth}
			\centering
			\begin{minipage}[][][b]{0.32\textwidth}
				\centering
				\subfigure[lof][Original]{\includegraphics[height=4cm]{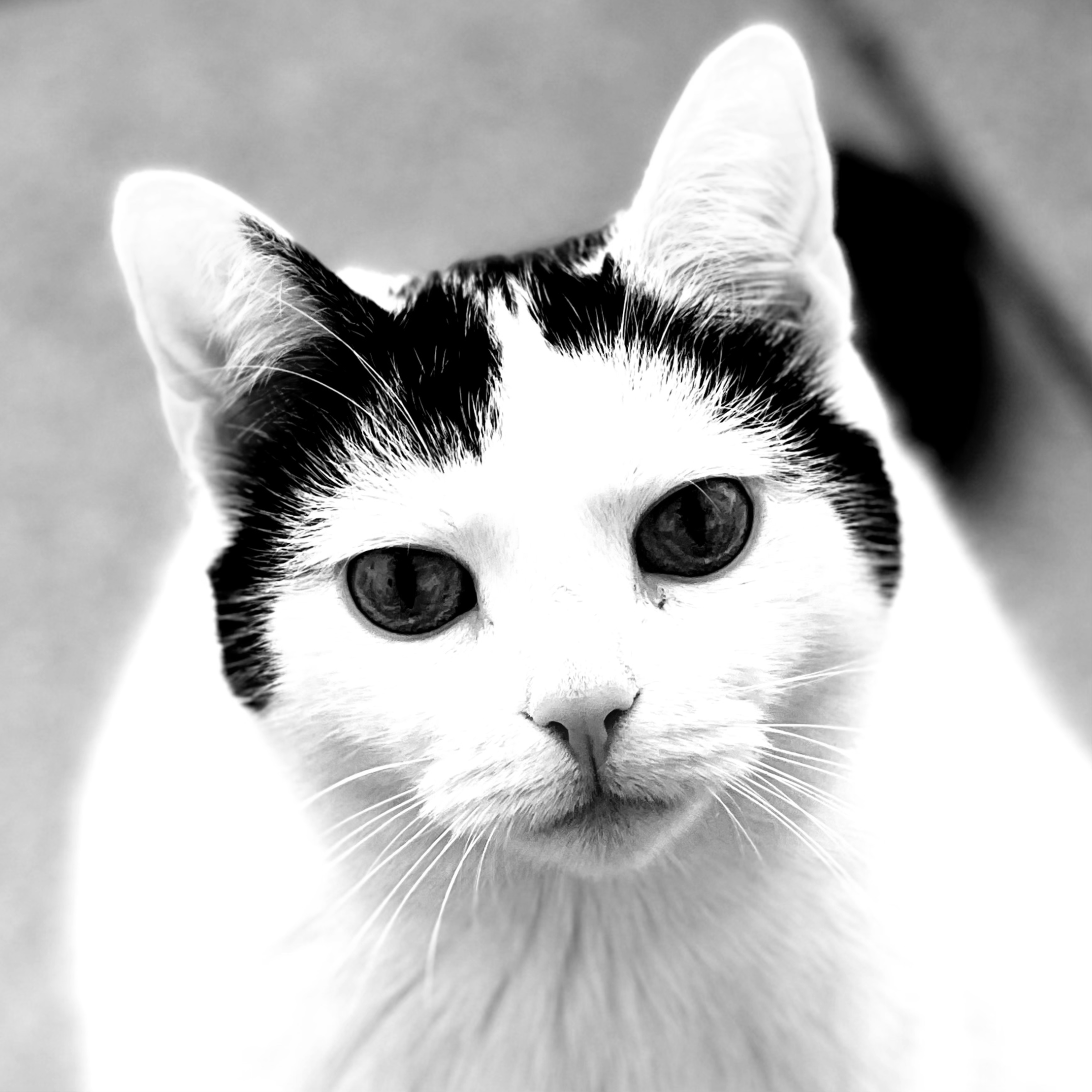}}
				\subfigure[lof][Noisy]{\includegraphics[height=4cm]{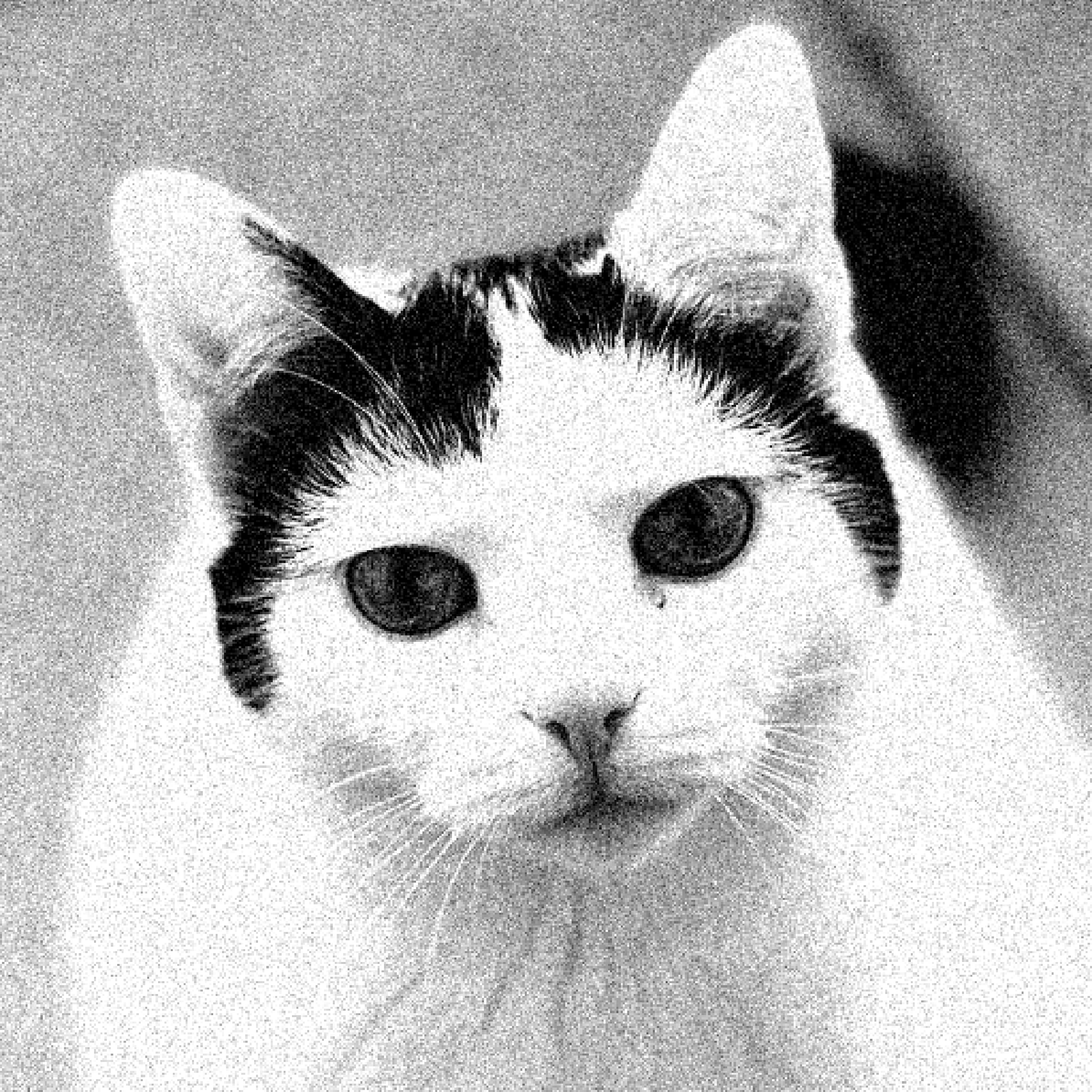}}
			\end{minipage}
			\begin{minipage}[][][b]{0.32\textwidth}
				\centering
				\subfigure[lof][$c = \sqrt{2}$]{\includegraphics[height=4cm]{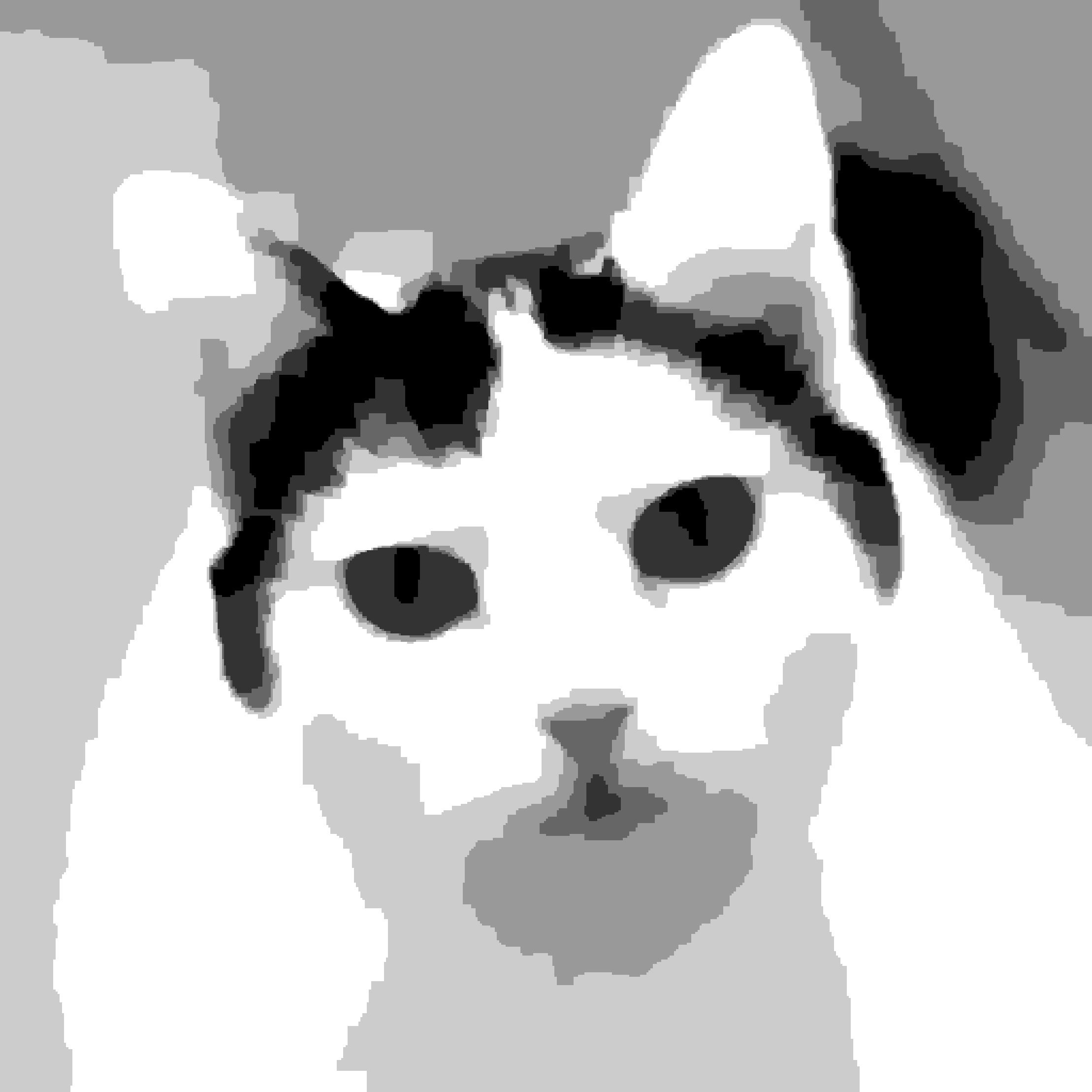}}
				\subfigure[lof][$c = 3 \sqrt{2}$]{\includegraphics[height=4cm]{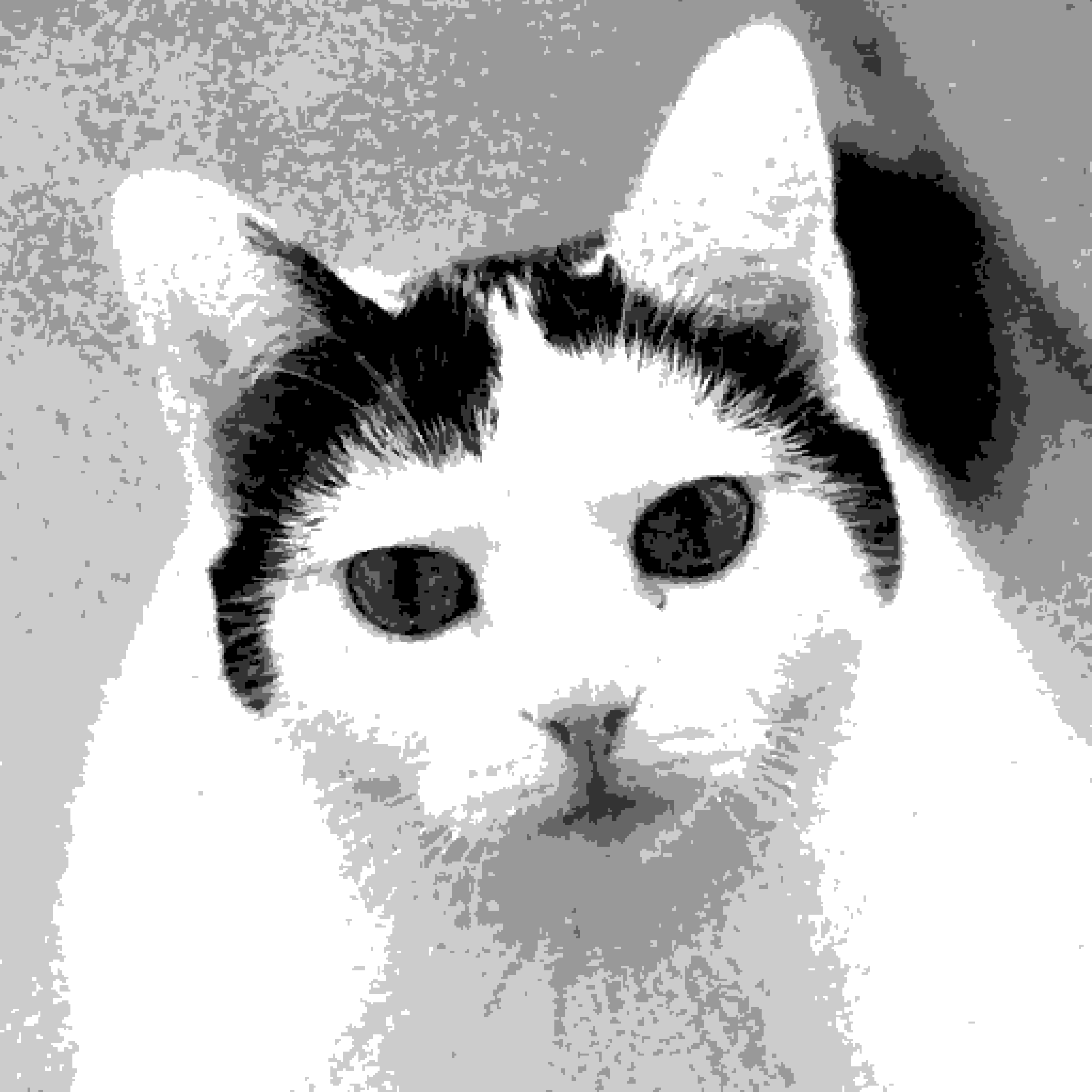}}
			\end{minipage}
			\begin{minipage}[][][b]{0.32\textwidth}
				\centering
				\subfigure[lof][$c = 9  \sqrt{2}$]{\includegraphics[height=4cm]{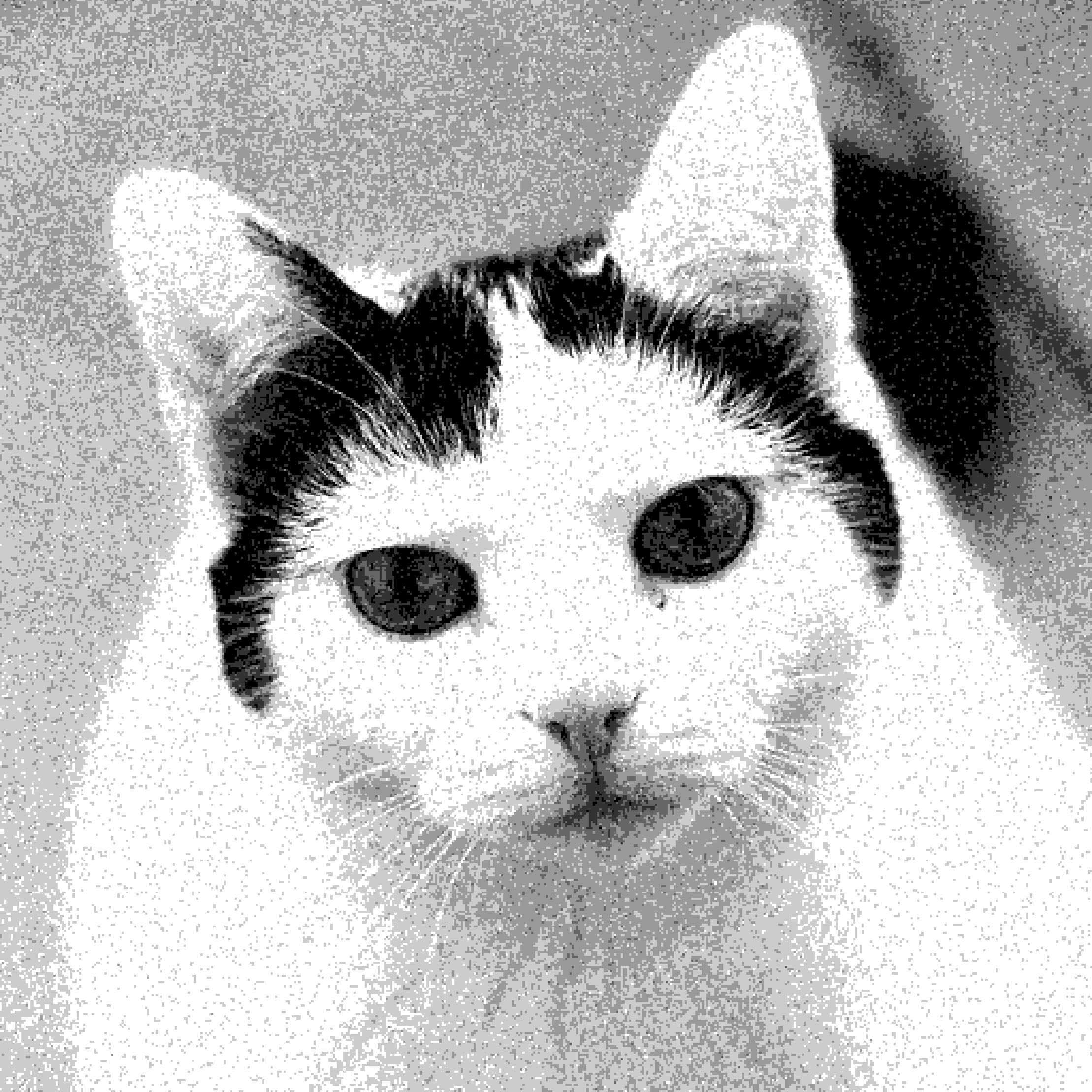}}
				\subfigure[lof][$c = 27 \sqrt{2}$]{\includegraphics[height=4cm]{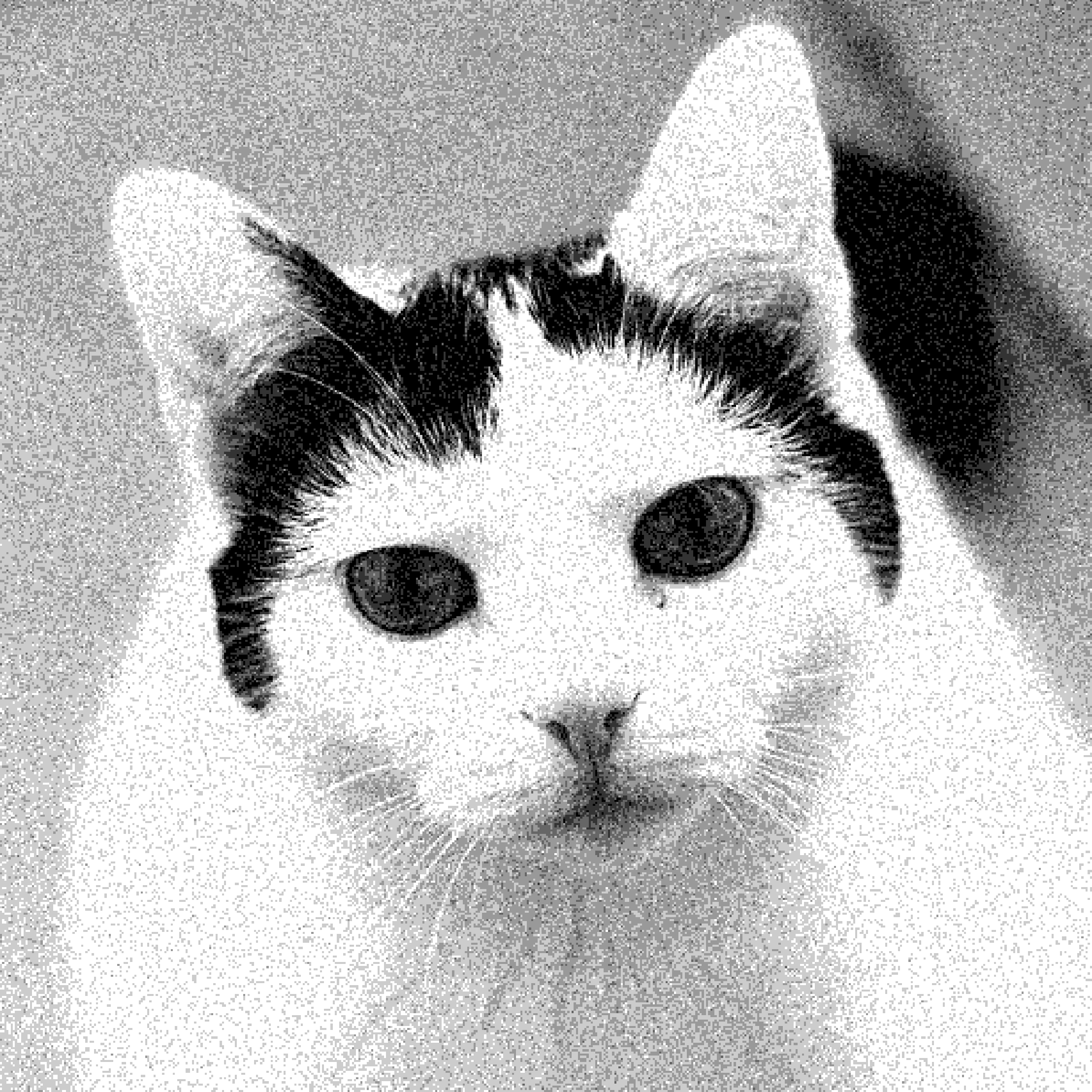}}
			\end{minipage}
		\end{minipage}
		\caption{Original image, noisy image, and resulting images obtained by applying \cref{alg:outer_approx} to the discretization of \eqref{eq:P_I} with $(h,\tau) = (\frac{1}{64},\frac{1}{512})$ and different values $c$.}
		\label{fig:cats}
	\end{figure}
	\begin{table}[t]
		\centering
		\caption{Results from the application of \cref{alg:outer_approx} to \eqref{eq:P_c^h} resulting from \eqref{eq:P_I} with $(h,\tau) = (\frac{1}{32},\frac{1}{128})$.}
		\label{fig:table_cats_32_4}	
		\resizebox{\textwidth}{!}{
			\begin{tabular}{ccccccccc}
				\toprule
				$c$ & Term. & It. & Obj. val. & $\TV$ & $\TVh$ & $V$ & Gap & Time (s) \\
				\midrule
				$\sqrt{2}$ & Opt & 1 & 0.453581 & 34.546875 & 19.985456 & 24.428330 & \num{-2.223e-01} & 28 \\
				$3 \sqrt{2}$ & Tol & 24 & 0.328303 & 89.679688 & 21.143576 & 21.137705 & \num{2.777e-04} & 7613 \\
				$9 \sqrt{2}$ & MaxIter & 25 & 0.297379 & 135.218750 & 18.054645 & 17.843617 & \num{1.169e-02} & 2154 \\
				$27 \sqrt{2}$ & MaxIter & 25 & 0.297393 & 135.132813 & 17.967840 & 17.798740 & \num{9.411e-03} & 1502 \\
				$81 \sqrt{2}$ & MaxIter & 25 & 0.297377 & 135.140625 & 17.973196 & 17.782075 & \num{1.063e-02} & 1786 \\
				\bottomrule
		\end{tabular}}
	\end{table}
	\begin{table}[t]
		\centering	
		\caption{Results from the application of \cref{alg:outer_approx} to \eqref{eq:P_c^h} resulting from \eqref{eq:P_I} with $(h,\tau) = (\frac{1}{64},\frac{1}{512})$.}
		\label{fig:table_cats_64_8}	
		\resizebox{\textwidth}{!}{
			\begin{tabular}{ccccccccc}
				\toprule
				$c$ & Term. & It. & Obj. val. & $\TV$ & $\TVh$ & $V$ & Gap & Time (s) \\
				\midrule
				$\sqrt{2}$ & GrbTime & 9 & 0.498596 & 29.820312 & 21.295309 & 21.086300 & \num{9.815e-03} & 343893   \\
				$3 \sqrt{2}$ & GrbTime & 13 & 0.437718 & 97.527344 & 24.093675 & 22.987415 & \num{4.591e-02} & 384429 \\
				$9 \sqrt{2}$ & MaxIter & 25 & 0.343889 & 301.912109 & 24.076626 & 23.721174 & \num{1.476e-02} & 539392 \\
				$27 \sqrt{2}$ & MaxIter & 25  & 0.307785 & 464.355469 & 21.445615 & 20.954688 & \num{2.289e-02} & 515671 \\
				$81 \sqrt{2}$ & Tol & 5 & 0.312697 & 491.566406 & 22.357736 & 22.357736 & \num{0.000e+00} & 35436 \\
				\bottomrule
		\end{tabular}}
	\end{table}
	\section{Conclusion and outlook} \label{sec:Outlook}
	We have introduced a two-level discretization of infinite-dimensional optimization problems with integrality constraints and total variation regularization, where the dual formulation of the total variation is discretized by means of Raviart--Thomas functions on a coarser mesh and the input function is discretized on an embedded mesh. A superlinear coupling of the mesh sizes has an averaging effect that enables to recover the total variation of an integer-valued function by the discretized total variation of an integer-valued recovery sequence. We added a constraint to the discretized problems that vanishes in the limit and which ensures the compactness of the sequence of minimizers of the discretized problems. Together, this ensures the convergence of these minimizers to a minimizer of the original problem. For the solution of the discretized problems, we introduced an outer approximation algorithm. We provided two numerical examples which confirm our theoretical results. As future work, we want to apply the developed discretization to the trust-region subproblems from \cite{manns2023integer} and solve them using acceleration techniques for integer programs.
	
	In \cref{sec:limsup}, we have argued that choosing the constant $c$ as small as possible suppresses chattering behavior for a fixed grid. This is confirmed in practice by the results for our second numerical example, see again \cref{fig:cats}. We believe that chattering is often undesired in practice and the motivation for using the $\TV$-regularizer in the first place. Since the available compute resources will generally limit the computationally feasible mesh sizes, we believe that choosing $c$ as small as theoretically justifiable is sensible in such situations. Consequently, we also believe that it is worthwhile to extend the proof of the two-dimensional case to the three-dimensional case in future research.
	
	\appendix
	\section{Proof of \Cref{thm:tvh_tv_d=2}} \label{appendix:A}
	We present the arguments that are needed for the proof of \Cref{thm:tvh_tv_d=2} and eventually the detailed proof of \Cref{thm:tvh_tv_d=2}. The first result is due to \cite{braides2017density} and states that each function in $\BVW(\Omega)$ can be approximated by a strictly converging sequence in $\BVW(\Omega)$ of functions with polygonal level sets.
	In line with \cite{braides2017density}, we first define polygonal sets for $d = 2$.
	\begin{definition}
		Let $d=2$. We say that a set $\Sigma \subset \Omega$ is \emph{polygonal} if there is a finite number of
		closed line segments $S_1,\dots,S_n \subset \R^2$ of strictly positive $\Ha^1$-measure
		such that $\Sigma$ coincides, up to $\Ha^1$-null sets,
		with $\bigcup_{j=1}^n S_j \cap \Omega$. We call a point in which at least two line segments intersect
		or a line segment ends a \emph{vertex}.
	\end{definition}
	\begin{lemma}\label{lem:approx_polygonal}
		Let $w \in \BVW(\Omega)$. There is a sequence $\{w^k\}_{k \in \N} \subset \BVW(\Omega)$ of functions with level sets with polygonal boundaries such that $w^k \to w$ in $L^1(\Omega)$ and $\TV(w^k) \to \TV(w)$.
	\end{lemma}
	\begin{proof}
		Follows from Theorem 2.1 and Corollary 2.5 in \cite{braides2017density}.
		Note that $\Omega$ is bounded in our setting and we do not need
		to work with $L^1_{loc}(\Omega)$ here.
	\end{proof}
	For $w \in \BVW(\Omega)$ with jumpsets with polygonal boundary, we are able to bound the limit of $\TV(R^W_{P0^\tau}(w))$ for $\tau \searrow 0$ from above by $\sqrt{2} \TV(w)$.
	\begin{theorem}\label{thm:TV_polygonal}
		Let $d = 2$, $\{\QQ_\tau\}_{ \tau >0}$ partitions of $\Omega$ fulfilling \cref{assu:grid}, and $w \in \BVW(\Omega)$ be such that the level sets of $w$ have polygonal boundaries.
		Then there holds
		\[\limsup_{\tau \searrow 0} \TV(R^W_{P0^\tau}(w)) \leq \sqrt{2} \TV(w).\]
	\end{theorem}
	\begin{figure}[h]
		\centering
		\resizebox{\textwidth}{!}{
			\begin{tikzpicture}
				\def\n{24} 
				\def\m{11} 
				\def\c{1}
				\def\d{55}
				\def\e{0.05}
				\def\f{0.071}
				\def\g{0.95}
				\def\h{0.929}
				\def\i{4.2}
				%
				\draw[step=0.5,gray!30,line width=1pt] (-1,-1) grid (\n,\m);
				%
				\draw[->, line width=1pt] (-0.5, 0) -- (\n-1.5,0) node[right] {$x_1$};
				\draw[->, line width=1pt] (0,-0.5) -- (0,\m-1.5) node[above] {$x_2$};
				\foreach \x in {1,2,...,22}
				\draw[line width=1pt] (\x,-0.1) -- (\x,0.1) node[below] {$\x$};
				\foreach \y in {1,2,...,9}
				\draw[line width=1pt] (-0.1,\y) -- (0.1,\y) node[left] {$\y$};
				\draw[black,line width=1pt,dash pattern=on 5mm off 2mm] (1,1) rectangle (\n-2,\m-2);
				\node[inner sep=\c pt, circle, fill] (1) at (1,6) {};
				\node[inner sep=\c pt, circle, fill] (2) at (5.5,1.5) {};
				\node[inner sep=\c pt, circle, fill] (3) at (18.5,3.5){};
				\node[inner sep=\c pt, circle, fill] (4) at (11.5,9){};
				\node[inner sep=\c pt, circle, fill] (5) at (22,7.5) {};
				\node at (12.5,4.5) {$w_1$};
				\node at (15,2) {$w_2$};
				\node at (18.5,7) {$w_3$};
				\node at (5.5,5.5) {$w_4$};
				\node at (5,7.5) {\Huge $\Omega$};			
				\draw (1) circle (\d pt);
				\draw (2) circle (\d pt);
				\draw (3) circle (\d pt);
				\draw (4) circle (\d pt);
				\draw (5) circle (\d pt);
				\draw[very thick] (1) -- (2);
				\draw[very thick] (2) -- (3);
				\draw[very thick] (3) -- (4);
				\draw[very thick] (4) -- (2);
				\draw[very thick] (3) -- (5);
				\draw[dotted, line width=1.5pt] (7,2) -- (10.5,2)
				-- (10.5,2.5)
				-- (13.5,2.5)
				-- (13.5,3);
				\draw[dotted, line width=1.5pt] (13.5,3) -- (17,3);
				\draw[dotted, line width=1.5pt] (6.5,3) -- (7,3)
				-- (7,3.5)
				-- (7.5,3.5)
				-- (7.5,4.5)
				-- (8,4.5)
				-- (8,5)
				-- (8.5,5)
				--(8.5,5.5)
				-- (9,5.5)
				-- (9,6)
				-- (9.5,6)
				-- (9.5,7)
				-- (10,7)
				-- (10,7.5)
				-- (10.5,7.5);
				\draw[dotted, line width=1.5pt] (13,8) -- (13,7.5)
				-- (13.5,7.5)
				-- (13.5,7)
				-- (14.5,7)
				-- (14.5,6.5)
				-- (15,6.5)
				-- (15,6)
				-- (15.5,6)
				-- (15.5,5.5)
				-- (16.5,5.5)
				-- (16.5,5)
				-- (17,5)
				-- (17,4.5);
				\draw[dotted, line width=1.5pt] (2,4.5) -- (2.5,4.5)
				-- (2.5,4)
				-- (3,4)
				-- (3,3.5)
				-- (3.5,3.5)
				-- (3.5,3)
				-- (4,3)
				-- (4,2.5);
				\draw[dotted, line width=1.5pt] (19.5,5) -- (20,5)
				-- (20,5.5)
				-- (20.5,5.5)
				-- (20.5,6)
				-- (21,6);
		\end{tikzpicture}}
		\caption{Example for a function $w \in \BVW(\Omega)$ with level sets that have polygonal boundaries with the notations from the proof of \cref{thm:TV_polygonal}.
		}\label{fig:polygonal_sketch}
	\end{figure}
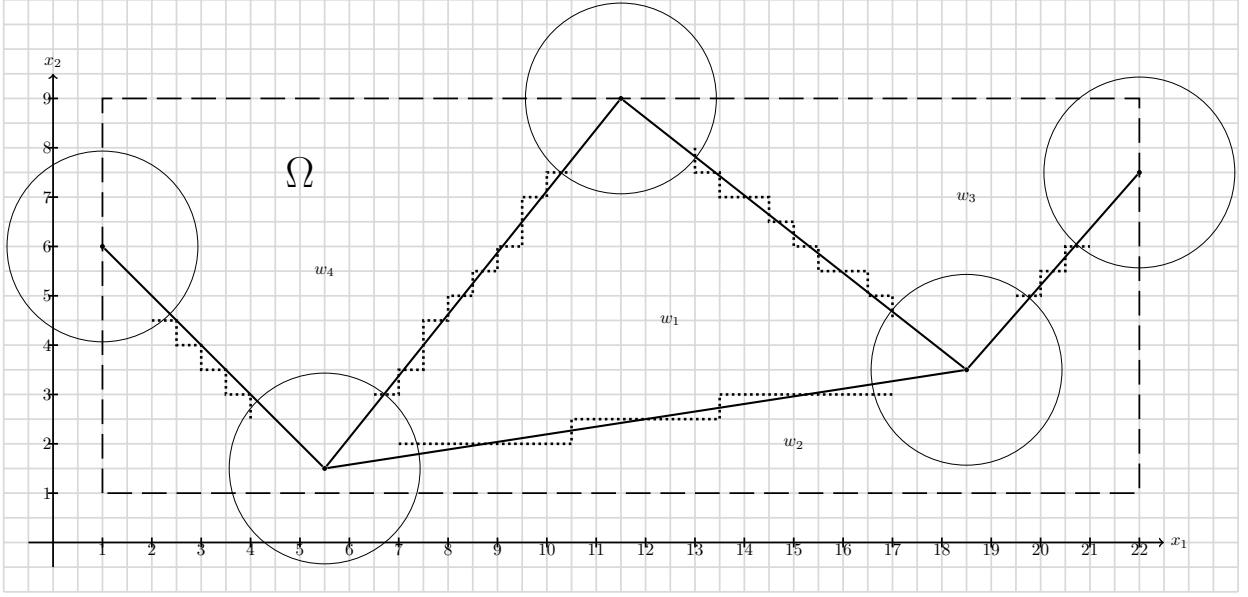
	\begin{proof}
		Let $w \in \BVW(\Omega)$ have level sets with polygonal boundaries consisting of $n \in \N$ line segments.
		We denote the set of these line segments by $\Scal$ and the associated set of vertices by $V$.
		We may assume without loss of generality that the line segments in $\Scal$ do not contain vertices in their relative
		interior, i.e., one line segment connects exactly two adjacent vertices. This yields that the number of vertices is
		also bounded by twice the number of line segments $n$. Moreover, a line segment $S \in \Scal$ separates exactly two level sets inside $\Omega$ and we may restrict the line segments to
		$\bar{\Omega}$.
		
		We make the following preliminary considerations: consider a ball $B_\eps(v)$ with radius $\eps > 0$ centered at a vertex $v \in V$. Denote by $\alpha \in (0,\pi]$ the angle between two segments $S_1, S_2 \in \Scal$ with $S_1 \neq S_2$ meeting in $v$. The distance $\delta \geq 0$ between the two points $\partial B_{\eps}(v) \cap S_1$ and $\partial B_{\eps}(v) \cap S_2$ is then given by $\delta = \eps s$, where we define $s \coloneqq 2 \sin\left(\frac{\alpha}{2}\right)$. Let the mesh size $\tau >0$ be small enough such that there hold $\partial B_\eps(v) \cap S_1 \neq \emptyset$ and $\partial B_\eps(v) \cap S_2 \neq \emptyset$ with $\eps \coloneqq \left( 1+ \frac{1}{s} \right) \sqrt{2} \tau$. Then there holds $\delta > \sqrt{2} \tau$, that is $\partial B_\eps(v) \cap S_1$ and $\partial B_\eps(v) \cap S_2$ do not lie in the same cube $Q \in \QQ_\tau$.
		The number of cubes of a given mesh $\QQ_\tau$ that intersect the ball $B_\eps(v)$ can be bounded from above by $\frac{4 (\eps+ \tau)^2}{\tau^2}$ so that we obtain for our choice of $\eps$ the upper bound $b \coloneqq 8\left( \frac{1}{\sqrt{2}} +1 + \frac{1}{s} \right)^2$ which is constant and independent of $\tau$ and $\eps$.
		
		Now let $\alpha \in (0,\pi]$ denote the minimum angle between two segments meeting in a vertex $v \in V$ and define $\eps$ as above. We consider the line segments $\Scal^\eps \coloneqq \left\{ S \setminus \cup_{v \in V} B_\eps(v) \mid S \in \Scal \right\}$ outside the balls $B_\eps(v)$. The number of segments in $\Scal^\eps$ is then for $\tau$ small enough equal to $n$. Moreover, by the choice of $\eps$, there is no cube $Q \in \QQ_\tau$ that contains more than one line segment from $\Scal^\eps$.
		
		We make the following observations. The reduced boundary of the 
		level sets of $R^W_{P0^\tau}(w)$ is by \Cref{dfn:RWP0tau} a subset of the boundaries of the cells of $\QQ_\tau$.
		Because the total variation of a $W$-valued function is the sum of the interface lengths weighted by their respective jump heights, we can split $\TV(R^W_{P0^\tau}(w))$ into a contribution 
		along segments from $\Scal^\eps$ and a contribution 
		inside the balls $B_\eps(v)$.
		
		We estimate the contribution inside the balls conservatively by multiplying the number of cubes inside the balls with their perimeter, that is in total by $4bn\tau$.
		
		Now let a segment $S = \conv\{u,v\} \in \Scal^\eps$ with $u,v \in \Omega$ be given. Then we can bound the contribution to $\TV(R^W_{P0^\tau}(w))$ along $S$ from above by $(\| u-v \|_{1} + 4 \tau) |w_1 -w_2| \leq (\sqrt{2} \| u-v \|_2 + 4 \tau)|w_1 - w_2|$, where $w_1,w_2 \in W$ denote the two values of $w$ within the two level sets separated by $S$. The estimate holds because the value of $R^W_{P0^\tau}(w)$ on a cube along $S$ is determined solely by the value of $w$ in the center point of the cube and by the monotonicity of the line segment $S$. Since the total variation of $w$ along $S$ is given by $|w_1 -w_2| \| u - v \|_2$, we conclude $\TV(R^W_{P0^\tau}(w)) \leq \sqrt{2} \TV(w) + 4 n (W_{\max} + b)\tau$ with $W_{\max} \coloneqq \max_{w_1,w_2 \in W} |w_1-w_2|$ so that $\limsup_{\tau \searrow 0} \TV(R^W_{P0^\tau}(w)) \leq \sqrt{2} \TV(w)$.
	\end{proof}
	Finally, we provide the proof of \cref{thm:tvh_tv_d=2}.
	\begin{proof}[Proof of \Cref{thm:tvh_tv_d=2}]
		By \cref{lem:approx_polygonal}, there is a sequence $\{ v^j \}_{j \in \N} \subset \BVW(\Omega)$ of functions with jump sets with polygonal boundaries such that $v^j \to w$ in $L^1(\Omega)$ and $\TV(v^j) \to \TV(w)$ as $j \to \infty$. We pick a subsequence $\{ v^{j_m} \}_{m \in \N}$ such that we have $\| v^{j_m} - w \|_{L^1(\Omega)} \leq \frac{1}{m}$ and $\TV(v^{j_m}) \leq \TV(w) + \frac{1}{m}$.
		By \cref{lem:estimate_RWP0h}, there holds $R^W_{P0^{\tau_h}}(v^{j_m}) \to v^{j_m} \text{ as } h \searrow 0 \text{ in } L^1(\Omega)$.
		Since the functions $\{v^{j_m}\}_{m \in \N}$ have jump sets with polygonal boundaries, we may apply \cref{thm:TV_polygonal} such that there holds $\limsup_{h \searrow 0} \TV(R^W_{P0^{\tau_h}} (v^{j_m})) \leq \sqrt{2} \TV(v^{j_m})$.
		This yields that for each $m \in \N$ there is some $h_m  >0$ such that
		$\| R^W_{P0^{\tau_h}} (v^{j_m}) - v^{j_m} \|_{L^1(\Omega)} \leq \frac{1}{m}$ and
		$\TV(R^W_{P0^{\tau_h}}(v^{j_m})) \leq \sqrt{2} \TV(v^{j_m}) + \frac{1}{m}$ for all $h \leq h_m$.
		We additionally choose the sequence $\{h_m\}_{m \in \N}$ such that it decreases strictly monotonically.
		We define $\{\tilde{w}^h\}_{h >0}$ by $\tilde{w}^h \coloneqq v^{j_m}$ for $h \in (h_{m+1},h_m]$
		and $w^h \coloneqq R^W_{P0^{\tau_h}} (\tilde{w}^h)$ for $h >0$. There holds for $h \in (h_{m+1},h_m]$
		that
		\begin{align*}
			\| w^h - w \|_{L^1(\Omega)} & = \| R^W_{P0^{\tau_h}} (\tilde{w}^h) - w \|_{L^1(\Omega)} = \| R^W_{P0^{\tau_h}} (v^{j_m}) - w \|_{L^1(\Omega)} \\
			& \leq \| R^W_{P0^{\tau_h}} (v^{j_m}) - v^{j_m} \|_{L^1(\Omega)} + \| v^{j_m} - w \|_{L^1(\Omega)} \leq \frac{2}{m},
		\end{align*}
		and
		\begin{align*}
			\TV(w^h) & = \TV(R^W_{P0^{\tau_h}} (\tilde{w}^h)) = \TV(R^W_{P0^{\tau_h}} (v^{j_m})) \\
			& \leq \sqrt{2} \TV(v^{j_m}) + \frac{1}{m} \leq \sqrt{2} \TV(w) + \frac{1+ \sqrt{2}}{m}.
		\end{align*}
		That is, for each $m \in \N$ there hold $\| w^h - w \|_{L^1(\Omega)} \leq \frac{2}{m}$ and $\TV(w^h) \leq \sqrt{2} \TV(w) + \frac{1 + \sqrt{2}}{m}$ for all $h \leq h_m$, which yields $w^h \to w$ in $L^1(\Omega)$ as $h \searrow 0$ and $\limsup_{h \searrow 0} \TV(w^h)$ $\leq \sqrt{2}\TV(w)$.
		In particular, the sequence $\{w^h\}_{h >0}$ is bounded in $\BV(\Omega)$, which together with the convergence in $L^1(\Omega)$ yields that $w^h \weakstarto w$ as $h \searrow 0$ in $\BV(\Omega)$.
		
		Moreover, by \Cref{lem:estimate_TVhRP0hTV}, for each $m \in \N$ there holds for all $h \in (h_{m+1},h_m]$
		\begin{align*}
			\TVh (w^h) & = \TVh(R^W_{P0^{\tau_h}}(v^{j_m})) \\
			& \leq \left(1+ \frac{4 \sqrt{2} \tau_h}{h}\right) \TV(v^{j_m}) \leq \left(1+ \frac{4 \sqrt{2} \tau_h}{h}\right) \left(\TV(w)+\frac{1}{m}\right),
		\end{align*}
		such that $\limsup_{h \searrow 0} \TVh (w^h) \leq \TV(w)$
		due to $\frac{\tau_h}{h} \searrow 0$ as $h \searrow 0$.
	\end{proof}
	
	\section*{Acknowledgments}
	The authors gratefully acknowledge computing time on
	the LiDO3 HPC cluster at TU Dortmund,
	partially funded in the Large-Scale Equipment
	Initiative by the Deutsche Forschungsgemeinschaft (DFG) as project 271512359.
	The authors thank two anonymous referees for providing helpful feedback on the article.
	\bibliographystyle{plain}
	\bibliography{biblio}
\end{document}